\documentclass[12pt]{amsart}

\usepackage{mathptmx} 
\usepackage{mathrsfs}
\usepackage{verbatim}
\usepackage{url}
\usepackage[all]{xy}
\usepackage[colorlinks=true,citecolor=blue]{hyperref} 

\usepackage{stmaryrd}
\usepackage{epsfig}
\usepackage{amsmath}
\usepackage{amssymb}
\usepackage{appendix}
\usepackage{amscd}
\usepackage{graphicx}
\usepackage{pstricks}

\allowdisplaybreaks

\theoremstyle{plain}
\newtheorem{thm}{Theorem}[section]
\newtheorem{lem}[thm]{Lemma}

\newtheorem{cor}[thm]{Corollary}
\theoremstyle{definition}

\newtheorem{exam}[thm]{Example}
\newtheorem{rem}[thm]{Remark}

\numberwithin{equation}{section}

\def \N {\mathbb N}
\def \Z {\mathbb Z}
\def \R {\mathbb R}
\def \U {\mathcal{U}}
\def \A {\mathcal A}

\def \L {\mathcal{L}}

\def \B {\mathcal B}

\def \V {\mathcal V}

\def \U {\mathcal{U}}

\def \a {\alpha }

\def \ep {\epsilon}

\def \C {\mathcal{C}}

\newtheorem{innercustomthm}{}
\newenvironment{customthm}[1]
{\renewcommand\theinnercustomthm{#1}\innercustomthm}
{\endinnercustomthm}

\DeclareMathOperator{\diam}{diam}
\DeclareMathAlphabet\mathbfcal{OMS}{cmsy}{b}{n}

\makeatletter
\@namedef{subjclassname@2020}{\textup{2020} Mathematics Subject Classification}
\makeatother

\begin{document}
		\title[Pressure for the average pseudo-orbits with block sub-additive potentials]{Pressure for the space of  average pseudo-orbits with block sub-additive potentials}
	\author{Fangzhou Cai}
	
	\author{Jie Li$^*$}\let\thefootnote\relax\footnote{* Corresponding author.}
	
	\address[F. Cai]{School of Mathematics and Systems Science, Guangdong Polytechnic Normal University, Guangzhou, 510665, PR China}
	\email{cfz@mail.ustc.edu.cn}

	\address[J. Li]{School of Mathematics and Statistics, Jiangsu Normal University, Xuzhou, Jiangsu 221116, China}
	\email{jiel0516@mail.ustc.edu.cn}


\begin{abstract}
 In this paper, we introduce the concept of block sub-additive potentials and study pressures on spaces of average pseudo-orbits for $\mathbb Z_+$-actions generated by continuous surjections on compact metric spaces. Relative to each finite open cover, we define topological and measure-theoretic pressures and establish a local variational principle. We prove that the local topological pressure coincides with the local pressure of the induced sub-additive potential. For every invariant measure, the local measure-theoretic pressure is bounded above by the measure-theoretic pressure of the induced potential, and equality holds when the measure is ergodic. Moreover, the global topological and measure-theoretic pressures associated with block sub-additive potentials are shown to be the suprema of their local counterparts over finite open covers. Consequently, the global topological pressure coincides with the pressure of the induced sub-additive potential, yielding a global variational principle.
\end{abstract}

\maketitle
	
\section{Introduction}		

Let $(X,T)$ be a {\it topological dynamical system} (abbr. {\it TDS})  in the sense that $X$  is a compact metric space endowed with a metric $d$ and $T$ is a continuous, surjective self-map on $X$.
When $T$ is a homeomorphism, the TDS $(X,T)$ is called {\it invertible}.
Entropy is an invariant that measures the complexity of a TDS.  The classical concepts of entropy include the measure-theoretic entropy for an invariant measure \cite{Kolomogorov58} and the topological entropy \cite{AKM65}. These two notions are related by the variational principle \cite{Gwyn,Gman}.
The study of entropy-like invariants and the variational principles relating their topological and measure-theoretic forms is a central part of entropy theory.

Originating from  ideas in statistical mechanics, Ruelle \cite{R73} introduced the notion of topological pressure for expansive systems,
which  Walters \cite{W75} later extended to the general context.
This topological pressure generalizes topological entropy and the corresponding variational principle as follows \cite{W75}:
\begin{equation}\label{equation-1}
  P(T,f) = \sup_{\mu\in M(X,T)}\Big\{h_\mu(T)+\int_X f \text{d}\mu\Big\},
\end{equation}
where $f\in C(X)$  (the space of all real-valued continuous functions on $X$ endowed with the supremum norm) is the {\it (additive) potential} on $X$, $P(T,f)$ represents the {\it topological pressure for $f$}, $M(X,T)$ denotes the space of all $T$-invariant Borel probability measures on $X$,  and $h_\mu(T)$ is the measure-theoretic entropy of $\mu$.
 
To investigate non-conformal repellers, Falconer \cite{Falconer88}
presented a thermodynamic formalism for sub-additive potentials, deriving  a variational principle subject to certain supplementary conditions on these potentials. In 2008, Cao, Feng and Huang  \cite{cfh} extended the variational principle  to sub-additive potentials in general TDSs.
Specifically, for a {\it sub-additive potential} $\mathcal{F}=\{f_n\}_{n=1}^\infty\subset C(X)$, defined by
$${f}_{m+n}(x)\le {f}_m(x)+f_n(T^mx)$$
 for any $x\in X$ and any positive integers $m, n$,
the equation \eqref{equation-1} can be generalized to the following:
\begin{equation}\label{equation-2}
  P(T,\mathcal{F}) = \sup_{\mu\in M(X,T)}\big\{h_\mu(T)+  \mathcal{F}_*(\mu)  \big\},
\end{equation}
where 
$$
\mathcal{F}_*(\mu)=\lim_{n\to \infty}	\frac{1}{n}\int f_n\ \text{d}\mu,
$$
the existence of limit follows from a sub-additive argument,
 and $P(T,\mathcal{F})$ denotes the {\it topological pressure for sub-additive potential $\mathcal{F}$}  defined via separated sets.
To avoid indeterminate expressions, when considering topological pressure and variational principles, we assume that $\mathcal{F}_*(\mu)>-\infty$ for every $\mu\in M(X,T)$. For results concerning a fixed invariant measure, this condition is required only for that measure.
The thermodynamic formalism for sub-additive potentials has been applied to Lyapunov exponents of matrix products and dimension theory
for non-conformal TDSs (see, e.g., \cite{Feng04, BCH10}).

Inspired by local entropy theory, Romagnoli \cite{R03} introduced measure-theoretic entropies for open covers and established a local variational principle. This principle was later generalized by Huang and Yi \cite{hy} to the pressure for additive potentials, and further extended to sub-additive potentials in \cite{zhang09, CDC10} along different lines. Specifically, equation \eqref{equation-2} can be localized to:
\begin{equation}\label{equation-3}
  P(T,\mathcal{F}, \U) = \sup_{\mu\in M(X,T)}\big\{h_\mu(T,\U)+  \mathcal{F}_*(\mu)  \big\},
\end{equation}
where $h_\mu(T,\U)$ denotes the measure-theoretic entropy of $\mu$ relative to the open cover $\U$. Analogous to \cite{ZC09}, we define
$$P_\mu(T,\mathcal{F}, \U):=h_\mu(T,\U)+ \mathcal{F}_*(\mu)$$
 as the {\it measure-theoretic pressure} for the sub-additive potential $\mathcal{F}$ and  open cover $\U$. The relationships between this pressure and other types of local measure-theoretic sub-additive pressures were investigated  in \cite{zhang09}.
 The  variational principles formulated in  \cite{hy, zhang09, CDC10} have been expanded in various directions,  including a relative local variational principle for sub-additive potentials \cite{MC13}, a conditional local variational principle for additive potentials \cite{SL23}, and a local variational principle for strongly sub-additive potentials under countable discrete amenable group actions \cite{ly}.

A related question is whether these local variational principles extend to general sofic group actions. 
Bowen \cite{Bowen10} introduced sofic measure entropy for actions of countable discrete sofic groups. Alternatively, from an operator algebra perspective, Kerr and Li \cite{kl11} developed a method to derive both measure-theoretic and topological entropy for general sofic actions.  Notably, these entropies are connected by a variational principle as in the classical case \cite[Section 9]{kl11}, and coincide with their classical counterparts when the  group is amenable \cite{Bowen12,kl13}.
Zhang  \cite{zhang12} further extended these results to local versions relative to open covers of the space.
Following the method proposed by Kerr and Li  \cite{kl11, kl13}, Chung \cite{chung} introduced the topological pressure for continuous actions of countable discrete sofic groups on compact metric space, demonstrating that the variational principle \eqref{equation-1} can be extended to the sofic context and, for amenable groups, the sofic topological pressure equates to the classical topological pressure.
 
The works in \cite{kl11} and \cite{chung} imply that the sofic entropy or pressure essentially measures  the complexity of (periodic) average pseudo-orbits relative to sofic approximation sequences. For the case of $\mathbb Z$, see \cite[Section 4]{Bowen12} and \cite[Section 2]{Bowen18}. Moreover, for results on entropy and pressure defined through pseudo-orbits of
$\mathbb Z$-actions, see \cite{BS90, cai, CL24, cl, HZ03} and their bibliographies.
 Notably, in our recent work \cite{CL24}, we used a measure-theoretic argument to prove that the topological pressure of average pseudo-orbits for additive potentials coincides with the classical topological pressure, corresponding to a special case of \cite[Theorem 1.1]{chung}.
In this paper, we develop a local pressure theory for block sub-additive potentials on the space of average pseudo-orbits for
$\mathbb Z_+$-actions generated by continuous surjections. This extends the $\mathbb Z$-action model-space approach to finite open covers and non-additive weights. The passage to non-invertible systems requires a further pressure comparison (see Lemma~\ref{lem:lift}).

\smallskip
 In general, an average pseudo-orbit  allows a minor average deviation over the long run of iterations (see subsection \ref{subsect:average_pseudo_orbit}).
To assign weights to finite sequences, we introduce the {\it block sub-additive potential} $\mathbfcal{F}=\{{\bf f}_n: {\bf f}_n\in C(X^n)\}_{n=1}^\infty$,  defined by the sub-additive inequality: $${\bf f}_{n+m}(x_0,x_1,\ldots,x_{n+m-1})\leq {\bf f}_n(x_0,\ldots,x_{n-1})+{\bf f}_m(x_{n},\ldots,x_{n+m-1})$$
for all $(x_0,\ldots,x_{n+m-1})\in X^{n+m}$ and all positive integers $m,n$.  This concept is indeed motivated by the research \cite{Y03, K04, FH10, C20} on weak Gibbs measures, equilibrium states, and quasi-Bernoulli measures, where cylinder functions (solely dependent on initial finite coordinates)  were used to produce special sub-additive potentials on full shift symbolic spaces. 
Let $\sigma$ denote the left shift on $X^\N$, and put
$$
\widehat f_n((x_i)_{i=0}^\infty) ={\bf f}_n(x_0,\ldots,x_{n-1}).
$$
Then block sub-additivity is equivalent to
$$
\widehat f_{n+m}\le\widehat f_n+\widehat f_m\circ\sigma^n.
$$
We use the block formulation to assign weights to finite sequences, including those which are not real orbit segments.
 Like an additive potential $f\in C(X)$, a block sub-additive potential can be defined independently of $T$, whereas the sub-additivity condition for a sequence in $C(X)$ involves $T$.
 Given a block sub-additive potential $\mathbfcal{F}=\{{\bf f}_n: {\bf f}_n\in C(X^n)\}_{n=1}^\infty$, its restriction to real orbit space produces  naturally a  sub-additive potential $\mathcal{F}=\{f_n: f_n\in C(X)\}_{n=1}^\infty$, defined by
$$f_n(x)={\bf f}_n(x,Tx,\ldots,T^{n-1}x),\quad \forall\  x\in X.$$
We write $\mathcal{F}$ for the induced potential and $\mathbfcal{F}$
for the block sub-additive potential. Conversely, every sub-additive potential admits a block sub-additive extension (see Lemma~\ref{lem:block-extension}), but such an extension need not be unique. We compare the local pressures on average pseudo-orbits with those of the induced potential. For the upper estimates, sub-additivity allows us to work with blocks of a fixed
length, where continuity gives the required estimates. The resulting topological pressure is determined by the induced potential, as is the
measure-theoretic pressure for an ergodic invariant measure. However, for a non-ergodic measure, different extensions of the same potential may give different pressures (see Example~\ref{ex:nonergodic-pressure}).

Following \cite{zhang12, chung}, we define the topological and measure-theoretic pressure, denoted by $GP(T, \mathbfcal{F}, \U)$ and $GP_\mu(T, \mathbfcal{F}, \U)$, respectively, for a block sub-additive potential $\mathbfcal{F}$ and an open cover $\U$ (see subsections \ref{subsect:topological-gsp} and \ref{subsect:measure-gsp}).
Analogous to \cite[Theorem 6.1]{kl11}, as well as \cite[Theorem 4.1]{zhang12} and \cite[Theorem 1.2]{chung}, we prove a local variational principle for these pressures as equation \eqref{equation-3}.

Denote $\C_X^o$ by the collection of all open covers of $X$.
\begin{thm}\label{thm:local variational principle}
Let $(X,T)$ be a TDS, $\mathbfcal{F}=\{{\bf f}_n\}_{n=1}^\infty$ be a block sub-additive potential and $\U\in \C_X^o$. Then
	$$ GP(T,\mathbfcal{F},\U)=\max_{\mu\in M(X,T)}GP_\mu(T,\mathbfcal{F},\U).$$
\end{thm}

The next theorem compares $GP_\mu(S,\mathbfcal{F},\U)$ with the pressure of the induced potential.

\begin{thm}\label{thm:measure-relation}
Let $(Y,S)$ be a TDS, $\mathbfcal{F}=\{{\bf f}_n\}_{n=1}^\infty$ be  a block sub-additive potential, $\mu\in M(Y,S)$ and $\U\in \C_Y^o$. Then for the sub-additive potential $\mathcal{F}$ induced by $\mathbfcal{F}$, we have
 $$ GP_\mu (S,\mathbfcal{F},\U)\leq P_\mu(S,\mathcal{F}, \U)= h_{\mu}(S,\U)+\mathcal{F}_*(\mu).$$
If additionally $\mu$ is ergodic, then
$$ GP_\mu (S,\mathbfcal{F},\U)=P_\mu(S,\mathcal{F}, \U)= h_{\mu}(S,\U)+\mathcal{F}_*(\mu).$$
\end{thm}
The comparison in Theorem~\ref{thm:measure-relation} is analogous to that for local Katok entropy \cite{K80}. 
The equality in Theorem~\ref{thm:measure-relation} may fail for non-ergodic invariant measures, even for an invertible system with two points; see Example~\ref{ex:nonergodic-pressure}.
The comparison in Lemma~\ref{lem:lift} allows us to deduce the general case of Theorem~\ref{thm:measure-relation} from the invertible case.

\medskip
For the topological case,  we show that the topological pressure of a block sub-additive potential  coincides with that of its induced sub-additive potential.
\begin{thm}\label{thm:topological-relation}
	Let $(X,T)$ be a TDS, ${\mathbfcal{F}}=\{{\bf f}_n\}_{n=1}^\infty$ be a block sub-additive potential and $\U\in \C_X^o$. Then for the sub-additive potential $\mathcal{F}$ induced by $\mathbfcal{F}$, we have
	$$GP(T,\mathbfcal{F},\U)= P(T,\mathcal{F},\U).$$
\end{thm}
\begin{proof} $ P(T,\mathcal{F},\U)\leq GP (T,\mathbfcal{F},\U)$ is clear from definitions. For the opposite side, by Theorem \ref{thm:local variational principle}, Theorem \ref{thm:measure-relation} and equation \eqref{equation-3}, there exists $\mu\in M(X,T)$ such that
$$ GP (T,\mathbfcal{F},\U)=GP_\mu(T,\mathbfcal{F},\U)\leq h_{\mu}(T,\U)+\mathcal{F}_*(\mu)\leq P(T,\mathcal{F},\U).$$
This completes the proof.
\end{proof}
Inspired by the work of Kerr and Li \cite{kl11}, we also give a direct topological proof of Theorem~\ref{thm:topological-relation} in Section~\ref{Sect:proof-thm3}.
This proof does not use the measure-theoretic comparison or the variational principles. We decompose an average pseudo-orbit into
blocks of a fixed length that are close to real orbit segments, leaving only a small proportion of coordinates uncovered. Block sub-additivity then gives the required estimates for the weighted covering sums.

\medskip
As demonstrated in \cite[Chapter 9]{W82} and \cite{cfh}, we can similarly introduce global pressure notions $GP(T,\mathbfcal{F})$ and $GP_\mu(X,\mathbfcal{F})$ for invariant measure $\mu$,  defined through separated sets for block sub-additive potentials (see subsections \ref{subsect:topological-gsp} and \ref{subsect:measure-gsp}).

The following theorem establishes a connection between the local and global cases.
\begin{thm}\label{thm:global-local}
Let $(X,T)$ be a TDS, and  $\mathbfcal{F}=\{{\bf f}_n\}_{n=1}^\infty$ be a block sub-additive potential. Then	
$$GP(T,\mathbfcal{F})=\sup_{\U\in \mathcal{C}^o_X}GP (T,\mathbfcal{F},\U),$$
and for every $\mu\in M(X,T)$, 
$$GP_\mu(T,\mathbfcal{F})=\sup_{\U\in \mathcal{C}^o_X}GP_\mu (T,\mathbfcal{F},\U).$$
\end{thm}

As a corollary, we derive a global variational principle.
\begin{cor}
	Let $(X,T)$ be a TDS and $\mathbfcal{F}=\{{\bf f}_n\}_{n=1}^\infty$ be a block sub-additive potential. Then for the sub-additive potential $\mathcal{F}$ induced by $\mathbfcal{F}$, we have
\begin{enumerate}
\item
	$$GP(T,\mathbfcal{F})=P(T,\mathcal{F})=\sup_{\mu\in M(X,T)}\big\{h_\mu(T)+  \mathcal{F}_*(\mu)  \big\}.$$
\item For $\mu\in M(X,T)$, we have
 $$GP_\mu(T,\mathbfcal{F})\leq h_\mu(T)+\mathcal{F}_*(\mu).$$
 If $\mu$ is ergodic, then
 $$GP_\mu(T,\mathbfcal{F})=h_\mu(T)+\mathcal{F}_*(\mu).$$
\end{enumerate}
\end{cor}
\begin{proof}
(1) As stated by \cite[Proposition 4.8]{cfh}, $P(T,\mathcal{F})=\sup_{\U\in \C_X^o} P(T,\mathcal{F}, \U)$. 
Then the desired conclusion follows from Theorems~\ref{thm:topological-relation} and \ref{thm:global-local}.
 
(2) From \cite[Lemma 2.3]{HYZ06}, we have $h_\mu(T)=\sup_{\U\in \C_X^o} h_\mu(T, \U)$.  
 Then the desired conclusion follows from Theorems~\ref{thm:measure-relation} and \ref{thm:global-local}.
\end{proof}

Our results extend the pressure comparison for average pseudo-orbits in \cite{CL24} to block sub-additive potentials and finite open covers. They also give, for $\mathbb Z_+$-actions, local sub-additive counterparts of the $\mathbb Z$-action cases of \cite{kl11,zhang12,chung}.
The comparison with the induced potential relates these pressures to the variational principles in \cite{cfh,CDC10,hy,zhang09} and to the
$\mathbb Z$-action case of \cite{ly}.

\medskip

The paper is structured as follows. Section \ref{Sect:preliminaries} cotains relevant notations and introduces the concept of pressure for  block sub-additive potentials. Sections \ref{Sect:proof-thm1} to \ref{Sect:proof-thm4} give the proofs of Theorems \ref{thm:local variational principle} to \ref{thm:global-local}, respectively.

\section{Preliminaries}\label{Sect:preliminaries}

Throughout this paper, we denote $\N$, $\Z_+$, $\Z$ and $\R$,  respectively, by  the sets of positive integers,
non-negative integers, integers and real numbers.

\subsection{Product spaces}		 Given a TDS $(X,T)$ with metric $d$,   the infinite product space $X^\N$  is compact and metrizable under the product topology, with metrization provided by the product metric $\tilde{d}$, defined as $$\tilde{d}\big((x_i)_{i=0}^\infty,(y_i)_{i=0}^\infty\big)=\sum_{i=0}^{\infty}\frac{d(x_i,y_i)}{2^i},$$
where $(x_i)_{i=0}^\infty,(y_i)_{i=0}^\infty$ are two points in $X^\N.$
Consider the shift transformation $\sigma: X^\N\to X^\N$, given by $$\sigma\big((x_0,x_1,\ldots,x_i,\ldots)\big)=(x_1,x_2,\ldots,x_{i+1},\ldots)$$
 for any $(x_0,x_1,\ldots,x_i,\ldots)\in X^\N$.
It is clear that the pair $(X^\N, \sigma)$ becomes a TDS.

For fixed $n\in\N$, write  $X^n$ as the $n$-fold product space $X\times X\times\dots\times X$.
The metric $d_n$ on $X^n$ is defined by
$$
d_n((x_0,x_1,\dots,x_{n-1}), (y_0, y_1,\dots,y_{n-1}))=\max_{0\le i\le n-1} d(x_i, y_i),
$$
for any $(x_0,x_1,\dots,x_{n-1}), (y_0, y_1,\dots,y_{n-1})\in X^n$.
 
\subsection{Induced spaces}\label{subsec:induced-spaces}
Let $X$ be a compact metric space with metric $d$. The induced hyperspace $2^X$, consisting of  all non-empty closed subsets of $X$, is compact and metrizable under the  Vietoris topology. A compatible metric for this topology is the Hausdorff metric $d_H$, defined as: for $A,B\in 2^X$,
\begin{equation*}
	\begin{split}
		d_H(A,B)&=\max\{\max_{a\in A}d(a,B),\max_{b\in B}d(b,A)\}\\
  & =\inf \{\ep>0:A\subset B_\ep(B)\text{\ and\ }B\subset B_\ep(A)\},
	\end{split}
\end{equation*}
 where $d(b, A)=\min_{a\in A} d(b,a)$ for each $b\in B$, and
$B_\ep(A)=\cup_{a\in A} B_d(a,\ep)$.

\medskip
Let $M(X)$ be the set of all Borel probability measures on $X.$
The weak$^*$-topology	on $M(X)$ is the weakest topology making each of the  maps $\mu\mapsto \int f \text{d}\mu$ continuous for every $f\in C(X).$
 It is known that $M(X)$ is compact and metrizable under weak$^*$-topology.

 We say that  a function $f\in C(X)$ is  {\it Lipschitz} if there exists a constant $L$ such that $|f(x)-f(y)|\leq L\ d(x,y)$ for all $x,y\in X.$ Set
 $$p_L(f)=\sup_{x\neq y}\frac{|f(x)-f(y)|}{d(x,y)}.$$
It is clear that $f$ is a Lipschitz function if and only if $p_L(f)< \infty$.
 For $\mu,\nu\in M(X),$ define
$$D(\mu,\nu)=\sup_{f\in C(X),p_L(f)\leq 1}\Big|\int_X f \text{d}\mu-\int_X f \text{d}\nu\Big|.$$
Then $D$	is a metric  and  compatible with
the weak*-topology on $M(X)$(see, e.g., \cite[Corollary 21.2.4]{Garling07}).
For $x\in X$, define $\delta_x\in M(X)$ as the Dirac measure supported on $\{x\}$.
It is not hard to check that
$$D(\delta_x,\delta_y)\leq d(x,y),\ \forall\  x, y\in X.$$

Let $(X,T)$ be a TDS and $\mu\in M(X).$ Define the induced map,  still denote by $T$, from  $M(X)$ to itself, as  $T\mu=\mu\circ T^{-1}$ for every $\mu\in M(X)$.
We say that $\mu\in M(X)$ is  {\it $T$-invariant} if $T\mu=\mu.$  Denote by $M(X,T)$  the collection of  all $T$-invariant
Borel probability measures on $X$.  A $T$-invariant $\mu\in M(X,T)$ is called  {\it ergodic} if, for each Borel  subset $A\subset X$ satisfying $T^{-1}A=A$, we have either $\mu(A)=0$ or $\mu(A)=1$.

\subsection{Covers}	
Let $(X,T)$ be a TDS and $\B(X)$ be the collection of all Borel subsets of $X$.  For a subset $K\subset X$, let $\B(K)$ denote the restriction of $\B(X)$ to $K$.
A  {\it cover} of $K$ is referred to  a finite sub-family of $\B(K)$, whose union equals $K$.
An {\it open cover} of $K$ consists of open sets in $K$ that together cover $K$.
A {\it partition} of $K$ is a cover of $K$  where the elements  are pairwise disjoint.
Denote by $\mathcal{C}_K$,  $\mathcal{C}^o_K$ and $\mathcal{P}_K$, respectively,  the set of all  covers,   open covers and  partitions of $K$.

Let $\U,\V\in \mathcal{C}_X$. We say that $\V$ is {\it finer} than $\U$, denoted as $\V\succeq\U$,  if  each element of
	$\V$ is contained in some element of $\U.$
Define  the {\it join} of $\U$ and $\V$ by
	$$\U\vee\V=\{U\cap V:U\in\U,V\in\V\}.$$
For $n\in\N$, simply write
	$$\U^{n-1}_0=\bigvee_{j=0}^{n-1}T^{-j}\U.$$

Given $\U\in \mathcal{C}_X$  and $K\subset X$, define $\U|_{K} =\{U\cap K:U\in\U\}$. Then $\U|_{K}\in \mathcal{C}_K.$
If  in addition $\U\in \mathcal{C}^o_X$,  then $\U|_{K}\in \mathcal{C}^o_K.$
{\bf For $\V\in \mathcal{C}_K$ and $\U\in \mathcal{C}_X$, when the context is clear, we simply reuse $\V\succeq\U$ to denote $\V\succeq\U|_{K}$.}
Define $N(\U,K)$ as the smallest cardinality of sub-families of $\U$ covering $K$, and write $N(\U)$ as $N(\U,X)$ for brevity.
Fix $n\in\N$, define
$$\U^n=\{U_1\times\cdots\times U_n: U_i\in \U, i=1,\ldots,n\}.$$
Then $\U^n\in \mathcal{C}_{X^n}.$

\subsection{Local entropies}
Let $(X,T)$ be a TDS and $\mu\in M(X,T)$. For $\a\in \mathcal{P}(X)$, define
$$H_\mu(\a)=-\sum_{A\in\a}\mu(A)\log\mu(A).$$
The classical {\it measure-theoretic entropy of $(X,T)$ with respect to $\a$} is given by
$$h_{\mu}(T,\a)=\lim_{n\to\infty}\frac{1}{n}H_{\mu}(\a_0^{n-1}).$$
Fix $\U\in \mathcal{C}^o_X.$ According to Romagnoli \cite{R03}, the {\it measure-theoretic entropies relative to $\U$}, can be defined as follows:
$$H_\mu(\U)=\inf_{\a\succeq \U,\a\in \mathcal{P}_X}H_{\mu}(\a),$$
$$h_{\mu}(T,\U)=\lim_{n\to\infty}\frac{1}{n}H_{\mu}(\U_0^{n-1}),$$
and
$$h_{\mu}^+(T,\U)=\inf_{\a\succeq \U,\a\in \mathcal{P}_X}h_{\mu}(T,\a).$$

\subsection{Topological pressure with sub-additive potentials}
Let $(X,T)$ be a TDS and $\mathcal{F}=\{f_n\}_{n=1}^\infty$ be a sub-additive potential on $X$.
Given $n\in\N$ and $\ep>0$,  we say that a set $E\subset X$ is an {\it $(n,\ep)$-separated subset} of $X$ with
respect to $T$ if
$$\max_{0\leq i\leq n-1} d(T^
ix, T^iy)>\ep$$
 for any two different points $x,y\in E.$

According to \cite{cfh}, the {\it topological pressure for $\mathcal{F}$} can be stated as follows:
$$P_{n}(T,\mathcal{F},\ep)=\sup\{\sum_{x\in E}\text{e}^{f_n(x)}: E \text{\ is\ an\ }(n,\ep)\text{-separated\ subset\ of } X\},$$	
$$P(T,\mathcal{F},\ep)=\varlimsup\limits_{n\to\infty}\frac{1}{n}\log  P_{n}(T,\mathcal{F},\ep),$$
and
$$P(T,\mathcal{F})=\lim_{\ep\to 0}P(T,\mathcal{F},\ep).$$

\medskip
Fix $\U\in \mathcal{C}^o_X$. Based on the work  \cite{ly,CDC10},  the {\it topological pressure for $\mathcal{F}$ relative to $\U$} can be defined in the following manner:
$$P_{n} (T,\mathcal{F},\U)= \inf\{\sum_{V\in \V}\sup_{x\in V}\text{e}^{f_n(x)}:\V\in \mathcal{C}_{X},\V\succeq\U_0^{n-1}\},$$	
and
$$P (T,\mathcal{F},\U)=\lim\limits_{n\to\infty}\frac{1}{n}\log P_{n} (T,\mathcal{F},\U).$$

\subsection{Average pseudo-orbits}\label{subsect:average_pseudo_orbit}

Let $(X,T)$ be a TDS. Given $2\le n\in\N$ and $ \delta>0,$ a tuple $(x_0,\dots,x_{n-1})\in X^n$ is said to be an {\it $(n,\delta)$-pseudo-orbit} if
$$
d(Tx_i,x_{i+1})\le\delta,\ \  i=0, 1,\dots, n-2;
$$
and be an {\it $(n,\delta)$-average-pseudo-orbit} if
$$
\frac{1}{n-1}\sum_{i=0}^{n-2}d(Tx_i,x_{i+1})\le\delta.
$$
Let $X^n_{\delta}$ denote the space of all $(n,\delta)$-average-pseudo-orbits. For $n=1$, we set $X_\delta^1=X$.

\begin{rem}\label{rem:average-pseudo-orbit}
(1) Set $$E=\{0\le i\le n-2: d(Tx_{i}, x_{i+1})\le \sqrt{\delta}\}.$$
By Markov's inequality, if $(x_0,\ldots,x_{n-1})\in X^n_{\delta}$,
then the cardinality $|E|\ge (n-1)(1-\sqrt{\delta})$. Conversely, for any $(x_0,\ldots,x_{n-1})\in X^n$ satisfying  $|E|\ge (n-1)(1-\sqrt{\delta})$,   we have
$$\frac{1}{n-1}\sum_{i=0}^{n-2}d(Tx_i,x_{i+1})\le (\diam(X)+1)\sqrt{\delta},$$
showing that $(x_0,\ldots,x_{n-1})\in X^n_{(\diam(X)+1)\sqrt{\delta}}$. Here $\diam(\cdot)$ means the diameter of given set.

(2) It is of classical interest to study periodic average pseudo-orbits \cite{kl11, Bowen18}. A tuple $(x_0,\dots,x_{n-1})\in X^n$ is called an {\it $(n,\delta)$-periodic-average-pseudo-orbit} if
$$
\frac{1}{n}\Big(\sum_{i=0}^{n-2}d(Tx_i,x_{i+1})+d(Tx_{n-1},x_0)\Big)\le \delta.
$$
The two definitions are related as follows: every $(n,\delta)$-periodic-average-pseudo-orbit is an $(n,\frac{n}{n-1}\delta)$-average pseudo-orbit, and conversely, every $(n,\delta)$-average-pseudo-orbit is an $(n,\delta+\frac{\diam(X)}{n})$-periodic-average-pseudo-orbit.
We work with average pseudo-orbits below.
\end{rem}

\subsection{Topological pressure for block sub-additive potentials}\label{subsect:topological-gsp}
When ambiguity is absent, we write ${\bf x}=(x_0,x_1,\ldots,x_{n-1})\in X^n$  for convenience.

Let $(X,T)$ be a TDS, $\U\in \mathcal{C}^o_X$ and $\mathbfcal{F}=\{{\bf f}_n\}_{n=1}^\infty$ be a block sub-additive potential.
 For $n\in\N$ and $\delta>0$, put
\begin{equation}\label{equation-topo}
GP_{n,\delta} (T,\mathbfcal{F},\U)= \inf\{\sum_{V\in \V}\sup_{{\bf x}\in V}\text{e}^{{\bf f}_n({\bf x})}:\V\in \mathcal{C}_{X^n_\delta},\V\succeq\U^n\},
\end{equation}
and
$$GP_\delta (T,\mathbfcal{F},\U)=\varlimsup\limits_{n\to\infty}\frac{1}{n}\log GP_{n,\delta} (T,\mathbfcal{F},\U).$$
The {\it topological pressure for $\mathbfcal{F}$ relative to $\U$} is then given by
$$GP (T,\mathbfcal{F},\U)=\inf_{\delta>0}GP_\delta (T,\mathbfcal{F},\U).$$

\begin{rem}\label{rem:topo-GSP}
\eqref{equation-topo} is equivalent to
\begin{equation}\label{equation-topo-extra}
GP_{n,\delta} (T,\mathbfcal{F},\U)= \inf\{\sum_{B\in \beta}\sup_{{\bf x}\in B}\text{e}^{{\bf f}_n({\bf x})}:\beta\in \mathcal{P}_{X^n_\delta},\beta\succeq\U^n\}.
\end{equation}
 Indeed, each $\V=\{V_1,\dots,V_k\}\in \mathcal{C}_{X^n_\delta}$ with $\V\succeq\U^n$ generates a partition $\beta=\{B_i:1\le i\le k\}$ by $B_i= V_i\setminus (\cup_{j=1}^{i-1} V_j) $, $1\le i\le k$. It is clear that $\beta\succeq\V\succeq\U^n$ and
$$
\sum_{i=1}^k\sup_{{\bf x}\in V_i}\text{e}^{{\bf f}_n({\bf x})}\ge \sum_{i=1}^k\sup_{{\bf x}\in B_i}\text{e}^{{\bf f}_n({\bf x})}.
$$
Taking infimum over $\V$ gives \eqref{equation-topo-extra}.
\end{rem}

Analogously, we can introduce a global version of topological pressure for $\mathbfcal{F}$ via separated sets.
For $n\in\N$ and $\delta>0$,
a subset $E\subset X^n$   is said to be an {\it $(n,\ep)$-separated set}  if for each distinct points  $(x_i)^{n-1}_{i=0}\neq (y_i)^{n-1}_{i=0}\in E$, there exists $0\leq i\leq n-1$ such that $d(x_i,y_i)>\ep$.
Put
$$GP_{n,\delta}(T,\mathbfcal{F},\ep)=\sup\{\sum_{{\bf x}\in E}\text{e}^{{\bf f}_n({\bf x})}: E \text{\ is\ an\ }(n,\ep)\text{-separated\ subset\ of } X_\delta^n\},$$	
and
$$GP(T,\mathbfcal{F},\ep)=\inf_{\delta>0}\varlimsup\limits_{n\to\infty}\frac{1}{n}\log  GP_{n,\delta}(T,\mathbfcal{F},\ep).$$
Then the  {\it topological pressure for $\mathbfcal{F}$} is defined by
$$GP(T,\mathbfcal{F})=\lim_{\ep\to 0}GP(T,\mathbfcal{F},\ep).$$

\subsection{Measure-theoretic pressure for block sub-additive potentials}\label{subsect:measure-gsp}
For a set $Z$,  denote  $\L_{Z}$  by the collection of all finite subsets of $Z$.

Let $(X,T)$ be a TDS and $\mu\in M(X)$. By the Riesz representation theorem, $\mu$ corresponds to a positive linear functional of norm one on $C(X)$. We require the averages along a finite sequence to approximate $\int g\ \mathrm d\mu$ for $g$ in a finite set of continuous functions. For $L\in \L_{C(X)}$, $n\in\N$ and $\delta>0$, put
$$
X^n_L(\delta,\mu)=\Big\{(x_0,\ldots,x_{n-1})\in X^n: \max_{g\in L}\Big|\frac{1}{n}\sum_{i=0}^{n-1}g(x_i)-\int g\ \text{d}\mu\Big|\leq\delta\Big\},
$$
and
$$X^n_{\delta,\mu,L}=X_\delta^n\cap X^n_L(\delta,\mu).$$
Let $\U\in \mathcal{C}^o_X$ and  $\mathbfcal{F}=\{{\bf f}_n\}_{n=1}^\infty$ be a block sub-additive potential.
Define
\begin{equation}\label{equation-measure}
GP_{n,\delta,\mu,L} (T,\mathbfcal{F},\U)=\inf\{\sum_{V\in \V}\sup_{{\bf x}\in V}\text{e}^{{\bf f}_n({\bf x})}:\V\in \mathcal{C}_{X^n_{\delta,\mu,L}},\V\succeq\U^n\},
\end{equation}
with $GP_{n,\delta,\mu,L} (T,\mathbfcal{F},\U)=0$ if $X^n_{\delta,\mu,L}=\emptyset.$
Set
$$GP_{\delta,\mu,L} (T,\mathbfcal{F},\U)=\varlimsup\limits_{n\to\infty}\frac{1}{n}\log GP_{n,\delta,\mu,L} (T,\mathbfcal{F},\U)	$$
with the convention that $\log 0=-\infty.$
The {\it measure-theoretic pressure for $\mathbfcal{F}$ relative to $\U$} is then defined by
		$$GP_\mu (T,\mathbfcal{F},\U)=\inf_{L\in \L_{C(X)}}\inf_{\delta>0}GP_{\delta,\mu,L}(T,\mathbfcal{F},\U). $$		

\begin{rem}\label{rem:measure-pressure-GSP}
(1) Due to the reasoning in Remark \ref{equation-topo-extra}, \eqref{equation-measure} is equivalent to
\begin{equation}\label{eqation-measure-exta}
GP_{n,\delta,\mu,L} (T,\mathbfcal{F},\U)=\inf\{\sum_{B\in \beta}\sup_{{\bf x}\in B}\text{e}^{{\bf f}_n({\bf x})}:\beta\in \mathcal{P}_{X^n_{\delta,\mu,L}},\beta\succeq\U^n\}.
\end{equation}

(2) The definition of measure-theoretic pressure for $\mathbfcal{F}$ relative to $\U$ can be reformulated as follows:
Put
	$$X^n_{\delta,\mu}=\{(x_0,\ldots,x_{n-1})\in X_\delta^n: \frac{1}{n}\sum_{i=0}^{n-1}\delta_{x_i}\in \overline{B_D(\mu,\delta)}\},$$	
	where $\overline{B_D(\mu, \delta)}$ denotes the closure of the open ball $B_D(\mu, \delta)$. Since $M(X)$ is convex, we have  $\overline{B_D(\mu, \delta)}=\{\nu\in M(X): D(\mu,\nu)\le \delta\}$.
	Define
	$$GP_{n,\delta,\mu} (T,\mathbfcal{F},\U)=\inf\{\sum_{V\in \V}\sup_{{\bf x}\in V}\text{e}^{{\bf f}_n({\bf x})}:\V\in \mathcal{C}_{X^n_{\delta,\mu}},\V\succeq\U^n\},$$
with $GP_{n,\delta,\mu} (T,\mathbfcal{F},\U)=0$ if $X^n_{\delta,\mu}=\emptyset.$
Set
$$GP_{\delta,\mu} (T,\mathbfcal{F},\U)=\varlimsup\limits_{n\to\infty}\frac{1}{n}\log GP_{n,\delta,\mu} (T,\mathbfcal{F},\U)	$$
	with the convention that $\log 0=-\infty.$
Then we have
$$GP_\mu (T,\mathbfcal{F},\U)=\inf_{\delta>0}GP_{\delta,\mu}(T,\mathbfcal{F},\U). $$		

To observe this, simply note that for any given  $\delta>0$ and $L\in \L_{C(X)}$, by weak$^{*}$ continuity, there exists  a $\delta^\prime=\delta^\prime(\delta, L)\in (0,\delta)$ such that $X^n_{\delta^\prime,\mu}\subset X^n_{\delta,\mu, L}$. On the other hand, fix $x_*\in X$ and put
$$
\operatorname{Lip}_1^0(X) =\{f\in C(X):p_L(f)\le1,\ f(x_*)=0\}.
$$
This family is compact in the supremum norm by the Arzel\`a--Ascoli theorem. Since subtracting a constant does not change the difference of the integrals against two probability measures, the supremum defining $D$ can be taken over $\operatorname{Lip}_1^0(X)$. Given $\delta>0$, choose a finite $\delta$-net $L_\delta$ in this family. The triangle inequality then gives $X^n_{\delta,\mu,L_\delta}\subset X^n_{3\delta,\mu}$.
\end{rem}

Similarly, we can define a global measure-theoretic pressure for block sub-additive potentials by using separated sets.

Let $(X,T)$ be a TDS, $\mu\in M(X, T)$ and  $\mathbfcal{F}=\{{\bf f}_n\}_{n=1}^\infty$ be a block sub-additive potential.  Given $L\in\L_{C(X)}$, $n\in\N$ and $\delta>0$, put
 $$
GP_{n,\delta,\mu,L}(T,\mathbfcal{F},\ep)=\sup\{\sum_{{\bf x}\in E}\text{e}^{{\bf f}_n({\bf x})}: E \text{\ is\ an\ }(n,\ep)\text{-separated\ subset\ of } X_{\delta,\mu,L}^n\},
$$
and
$$GP_\mu(T,\mathbfcal{F},\ep)=\inf_{L\in \L_{C(X)}}\inf_{\delta>0}\varlimsup\limits_{n\to\infty}\frac{1}{n}\log  GP_{n,\delta,\mu,L}(T,\mathbfcal{F},\ep).$$
The {\it measure-theoretic pressure for $\mathbfcal{F}$} is defined by
$$GP_\mu(T,\mathbfcal{F})=\lim_{\ep\to 0}GP_\mu(T,\mathbfcal{F},\ep).$$

We end this section by showing that every sub-additive potential is
induced by a block sub-additive potential.

\begin{lem}\label{lem:block-extension}
Let $(X,T)$ be a TDS and $\mathcal{F}=\{f_n\}_{n=1}^\infty\subset C(X)$ be a sub-additive potential. Then there exists a block sub-additive potential $\mathbfcal{F}=\{{\bf f}_n\}_{n=1}^\infty$ such that
$$
{\bf f}_n(x,Tx,\ldots,T^{n-1}x)=f_n(x) \quad\text{for all }x\in X\text{ and }n\in\N.
$$
\end{lem}

\begin{proof}
For $n\in\N$ and $(x_0,\ldots,x_{n-1})\in X^n$, define
$$
{\bf f}_n(x_0,\ldots,x_{n-1}) =\min_{\substack{1\le k\le n\\ 0=r_0<r_1<\cdots<r_k=n}} \sum_{j=1}^k f_{r_j-r_{j-1}}(x_{r_{j-1}}),
$$
where all $r_j$ are integers. For each $n$, there are only finitely many such choices. Hence ${\bf f}_n\in C(X^n)$.

Let $n,m\in\N$ and $(x_0,\ldots,x_{n+m-1})\in X^{n+m}$. Choose decompositions attaining the minima for the first $n$ and
the last $m$ coordinates. Joining these decompositions yields an admissible choice in the definition of ${\bf f}_{n+m}$. Thus
\begin{align*}
{\bf f}_{n+m}(x_0,\ldots,x_{n+m-1})  \le {\bf f}_n(x_0,\ldots,x_{n-1}) +   {\bf f}_m(x_n,\ldots,x_{n+m-1}).
\end{align*}
Therefore $\mathbfcal{F}$ is block sub-additive.

For $x\in X$ and any integers $0=r_0<r_1<\cdots<r_k=n$, repeated use of sub-additivity gives
$$
f_n(x)\le \sum_{j=1}^k f_{r_j-r_{j-1}}(T^{r_{j-1}}x).
$$
Taking the minimum over these choices, we obtain
$$
f_n(x)\le {\bf f}_n(x,Tx,\ldots,T^{n-1}x).
$$
The choice $k=1$ gives the reverse inequality. This proves the lemma.
\end{proof}

The extension is not unique in general. Indeed, for $c>0$, put
$$
{\bf g}_n(x_0,\ldots,x_{n-1})={\bf f}_n(x_0,\ldots,x_{n-1}) -c\sum_{i=0}^{n-2}d(Tx_i,x_{i+1}),
$$
where the sum is taken to be zero for $n=1$.
The additional term when two blocks are joined is non-positive, so $\{{\bf g}_n\}_{n=1}^\infty$ is also block sub-additive.
Since all the distance terms vanish on real orbit segments, its induced potential is again $\mathcal{F}$.
For an example in which two extensions of the same potential have different measure-theoretic pressures for a non-ergodic invariant
measure, see Example~\ref{ex:nonergodic-pressure}.

\section{Proof of Theorem \ref{thm:local variational principle}}\label{Sect:proof-thm1}


\begin{customthm}{{\bf Theorem \ref{thm:local variational principle}}}
{\it Let $(X,T)$ be a TDS, $\mathbfcal{F}=\{{\bf f}_n\}_{n=1}^\infty$ be a block sub-additive potential and $\U\in \mathcal{C}_X^o$. Then
	$$ GP(T,\mathbfcal{F},\U)=\max_{\mu\in M(X,T)}GP_\mu(T,\mathbfcal{F},\U).$$}
\end{customthm}

\begin{proof}
By definitions, for any measure $\mu\in M(X,T)$, we have $GP_\mu(T,\mathbfcal{F},\U)\leq GP (T,\mathbfcal{F},\U)$. Therefore,  it suffices to show that there exists a measure $\mu\in M(X,T)$ such that $GP_\mu(T,\mathbfcal{F},\U)$ $ \geq GP(T,\mathbfcal{F},\U)$.

We first prove the following claim.

{\bf Claim:} For any $L\in\L_{C(X)}$ and any $\delta>0$, there exists $\nu\in M(X)$ such that
	$$GP_{\delta,\nu,L} (T,\mathbfcal{F},\U)\geq GP(T,\mathbfcal{F},\U).$$

{\it Proof of Claim}:  Fix $L\in\L_{C(X)}$ and $\delta>0$. Since $M(X)$ is compact under weak$^{*}$-topology,  we can find some $F\in \L_{M(X)}$ such that for any $n\in\N$ and $(x_i)_{i=0}^{n-1}\in X^n$, there exists a $\nu\in F$ such that
$$\max_{g\in L}\Big|\frac{1}{n}\sum_{i=0}^{n-1}g(x_i)-\int g\ \text{d}\nu\Big|\leq\delta.$$
 Hence we have
 $$\bigcup_{\nu\in F}X^n_{\delta,\nu,L}=X^n_\delta.$$
 This implies that for each $n\in\N,$ there exists $\nu_n\in F$ such that
  $$ GP_{n,\delta,\nu_n,L}(T,\mathbfcal{F},\U) \geq\frac{GP_{n,\delta}(T,\mathbfcal{F},\U)}{|F|}.$$
Since $F$ is finite, we can choose a subsequence $\{n_i\}_{i=1}^\infty$ such that
infinite sequence $\{n_i\}_{i=1}^\infty$ such  that
$$ \lim_{i\to\infty}\frac{1}{n_i}\log GP_{n_i,\delta}(T,\mathbfcal{F},\U)=GP_{\delta}(T,\mathbfcal{F},\U),$$
 and $\{\nu_{n_{i}}\}_{i=1}^\infty=\{\nu\}$ for some $\nu\in F$.
It is then derived that $$GP_{\delta,\nu,L}(T,\mathbfcal{F},\U)\geq GP_{\delta}(T,\mathbfcal{F},\U)\geq GP(T,\mathbfcal{F},\U).$$
This completes the proof of the Claim.

\medskip
	Now let $\{g_k\}_{k=1}^\infty$ be a sequence dense in $C(X)$ and  denote $L_m=\{g_k\}_{k=1}^m.$  By Claim, for any $m\in\N,$ there exists $\nu_m\in M(X)$ such that
	$$GP_{\frac{1}{m},\nu_m,L_m}(T,\mathbfcal{F},\U)\geq GP(T,\mathbfcal{F},\U).$$
	Passing to a subsequence if necessary, we assume that $\nu_{m}\to \mu\in M(X).$
	For any $L\in\L_{C(X)}$ and $\delta>0$, when $m$ is sufficiently large, the followings hold:
	\begin{enumerate}
		\item[(a)] $\frac{1}{m}<\frac{\delta}{4},$
		\item[(b)] for every $f\in L,$ there exists $g\in L_m$ such that $||f-g||_\infty<\frac{\delta}{4},$ and
		\item[(c)]   $|\int f\text{d}\nu_m-\int f\text{d}\mu|<\frac{\delta}{4}$ for all $f\in L.$
	\end{enumerate}
Then for any $n\in\N$, using the conditions (a)(b)(c),  it is not hard to verify that
$$
X^n_{\frac{1}{m},\nu_m,L_m}\subset X^n_{\delta,\mu,L}.
$$
Consequently,
$$GP_{\delta,\mu,L}(T,\mathbfcal{F},\U)\geq GP_{\frac{1}{m},\nu_m,L_m}(T,\mathbfcal{F},\U)\geq GP(T,\mathbfcal{F},\U).$$
Since $L\in\L_{C(X)}$ and $\delta>0$ are arbitrary, we have
$$GP_{\mu}(T,\mathbfcal{F},\U)\geq GP(T,\mathbfcal{F},\U).$$
	
 Now we are only left to prove that $\mu$ is $T$-invariant. For any $\ep>0$ and any Lipschitz function $f\in C(X)$ with $p_L(f)\leq 1$, we select $m$ to be sufficiently large such that  $\frac{1}{m}<\frac{\ep}{3}$, and
$$||f-g_1||_\infty<\frac{\ep}{3},\  \ ||f\circ T-g_2||_\infty<\frac{\ep}{3}$$
for some $g_1,g_2\in L_m$.
By the choice of $\nu_m$, we can choose integers $k_m\ge m+2$ such that $X^{k_m}_{1/m,\nu_m,L_m}\ne\emptyset.$
Let $(x_i)_{i=0}^{k_m-1}\in X^{k_m}_{\frac{1}{m},\nu_m,L_m}$. We have
$$\Big|\frac{1}{k_m}\sum_{i=0}^{k_m-1}g_1(x_i)-\int g_1\ \text{d}\nu_m\Big|\leq\frac{1}{m}<\frac{\ep}{3},$$
and
$$\Big|\frac{1}{k_m}\sum_{i=0}^{k_m-1}g_2(x_i)-\int g_2\ \text{d}\nu_m\Big|\leq\frac{1}{m}<\frac{\ep}{3}.$$
This implies that
$$\Big|\frac{1}{k_m}\sum_{i=0}^{k_m-1}f(x_i)-\int f \ \text{d}\nu_m\Big|<\ep,$$
 and
 $$\Big|\frac{1}{k_m}\sum_{i=0}^{k_m-1}f(Tx_i)-\int f\circ T\ \text{d}\nu_m\Big|<\ep.$$
Then
	\begin{equation*}
		\begin{split}
			\Big|\int f\ \text{d}\nu_m-\int f\ \text{d}  T\nu_m\Big|
			&\leq\Big|\frac{1}{k_m}\sum_{i=0}^{k_m-1}f(x_i)-	\frac{1}{k_m}\sum_{i=0}^{k_m-1}f(Tx_i)\Big|+2\ep\\
			&\leq \frac{1}{k_m}\sum_{i=0}^{k_m-2}d(Tx_i,x_{i+1})+\frac{2||f||_\infty}{k_m}+2\ep\\
			&\leq \ep+\frac{2||f||_\infty}{k_m}+2\ep.
		\end{split}
	\end{equation*}
	Letting $m\to\infty$ and given the arbitrariness of $\ep$, it follows that for every  Lipschitz function $f$ with $p_L(f)\leq 1$,   we have
$$\int f \ \text{d}\mu=\int f\ \text{d}T\mu.$$
Consequently, $\mu$ is $T$-invariant. This completes the proof.
\end{proof}

\section{Proof of Theorem \ref{thm:measure-relation}}\label{Sect:proof-thm2}

We first prove the Theorem  \ref{thm:measure-relation} for invertible systems and then deduce the general case.

\subsection{Proof of Theorem \ref{thm:measure-relation}: the  invertible case}
We begin with several lemmas.

	\begin{lem}[\cite{cfh}, Lemma 2.3]\label{subadd}
		Let $(X,T)$ be a TDS,  $\mathcal{F}=\{f_n\}_{n=1}^\infty$ be a sub-additive potential on $X$, and $\{\nu_n\}_{n=1}^\infty$  be a sequence in $M(X)$. Put $\mu_n= \frac{1}{n}\sum_{j=0}^{n-1}{\nu}_n\circ T^{-j}$, and assume that $\mu_{n_i}\to \mu$	
		in $M(X)$ for some subsequence $\{n_i\}_{i=1}^\infty\subset \N$. Then we have
		$$\varlimsup_{i\to \infty}\frac{1}{n_i}	\int  {f}_{n_i}\ \emph{d}{\nu}_{n_i}\leq\mathcal{F}_*(\mu).$$
	\end{lem}

\begin{lem}[\cite{W82}, Lemma 9.9]\label{lw}
	Let $a_1,\ldots,a_k\in \R$, $p_i\geq 0, i=1,\ldots,k$ and $\sum_{i=1}^kp_i=1$. Then
	$$\sum_{i=1}^kp_i(a_i-\log p_i)\leq \log(\sum_{i=1}^k\emph{e}^{a_i}),$$
and equality holds iff $p_i=\frac{\emph{e}^{a_i}}{\sum_{j=1}^k \emph{e}^{a_j}}$ for all $1\le i\le k$.
\end{lem}

	\begin{lem}[\cite{HYZ06}]\label{lb}
		Let $(X,T)$ be an invertible TDS, $\mu\in M(X,T)$ and  $\U\in \mathcal{C}^o_X$. Then $$h_{\mu}^+(T,\U)=h_{\mu}(T,\U).$$
	\end{lem}

Let $\U=\{U_1,\ldots,U_m\}\in  \mathcal{C}_X^o$. Define
$$\U^*= \{\a\in \mathcal{P}_X:\a=\{A_1,\ldots,A_{m}\}, A_i\subset U_i, i=1,\ldots, m \},$$
where $A_i$ can be empty for some $1\le i\le m$.

	\begin{lem}[\cite{hmry}, Lemma 2]\label{la}
		Let $(X,T)$ be a TDS, $\U=\{U_1,\ldots,U_m\}\in \mathcal{C}^o_X$  and $G:\mathcal{P}_X\to \R$ be monotone in the sense that $G(\a)\geq G(\beta)$ whenever $\a\succeq\beta.$ Then we have
		$$\inf_{\a\in\mathcal{P}_X,\a\succeq \U}G(\a)=\inf_{\a\in \U^*}G(\a).$$
	\end{lem}
	
Given a set $A\subset X$ and $\a\in \mathcal{P}_X$, write  $\partial A$ for  the boundary of $A$, and $\partial\a=\cup_{A\in\a}\partial A$.
 For $\mu\in M(X)$ and two ordered partitions $\a=\{A_1,\dots,A_m\}, \beta=\{B_1,\dots,B_m\}\in \mathcal{P}_X$, set $\mu(\a\Delta\beta)=\sum_{i=1}^m\mu(A_i\Delta B_i)$.

\begin{lem}\label{count}
		Let $(X,T)$ be a TDS, $\mu\in M(X,T)$ and $\U=\{U_1,\ldots,U_\text{m}\}\in \mathcal{C}_X^o$. Then we can find a countable family $\{\a_l\}_{l=1}^\infty$ in $\U^*$ such that $\mu(\partial\a_l)=0$ for each $l\in\N$, and
$$h_{\mu}^+(T,\U)=\inf_{l\in\N} h_{\mu}(T,\a_l).$$
\end{lem}
\begin{proof}
	Let $\{x_k\}_{k=1}^\infty$ be a dense subset of $X$. Since for each $x_i$, there is at most countably many  $r>0$ such that $$\mu(\{y\in X:d(y,x_k)=r\})>0.$$
Hence we can choose a sequence $\{r_n\}_{n=1}^\infty$ decreasing to $0$ such that
$$\mu(\partial B_{r_n}(x_k))=0,\ \forall k,\ n\in\N.$$
	Consider the set
$$\{\a\in\U^*: \a=\{V_1,\ldots,V_i-\cup_{j=1}^{i-1}V_j,\ldots,V_m-\cup_{j=1}^{m-1}V_j\}, V_i\text{\ is\ the\ union\ of\  finite\ }B_{r_n}(x_k)\}.$$
We  enumerate this countable family as $\{\a_l\}_{l=1}^\infty$.  It is clear that $\mu(\partial\a_l)=0$ for each $l\in\N$.

By Lemma \ref{la}, we have $h_{\mu}^+(T,\U)=\inf_{\a\in \U^*} h_{\mu}(T,\a).$  To prove  $\inf_{\a\in \U^*} h_{\mu}(T,\a)=\inf_{l\in\N} h_{\mu}(T,\a_l)$, it suffices to  show that for any $\beta\in\U^*$ and $\ep>0,$  there exists some $\a\in \{\a_l\}_{l=1}^\infty$	 such that $\mu(\a\Delta\beta)<\ep.$ Let $\beta=\{B_1,\ldots,B_m\}$
with  $B_i\subset U_i, \ i=1,\ldots,m$. For each $i=1,\ldots,m$, choose a compact subset $K_i\subset B_i$ with $\mu(B_i-K_i)<\frac{\ep}{2m^2}$.  Put
$$O_i=U_i\cap (\bigcup_{j\neq i }K_j)^c, i=1,\ldots,m.$$
Note that $B_i=(\bigcup_{j\neq i }B_j)^c\subset (\bigcup_{j\neq i }K_j)^c$, it follows that
$K_i\subset B_i\subset O_i$, and
$$\mu(O_i-K_i)= \mu((O_i-B_i)\cup(B_i-K_i))\le \sum_{i=1}^m\mu(B_i-K_i)<\frac{\ep}{2m}.$$
Furthermore, $\{O_1,\ldots,O_m\}$ forms an open cover of $X$. Let $\delta>0$ be the Lebesgue number of $\{O_i\}_{i=1}^m.$ Choose some $r_n<\frac{\delta}{2}$.
It follows that $\{B_{r_n}(x_k)\}_{k=1}^\infty$ covers $X$, and then there exists a finite subcover $\A$. Let $$V_i=\bigcup_{A\in\A,A\subset O_i}A, i=1,\ldots,m.$$
It is easy to see that $\{V_i\}_{i=1}^m$ is an open cover of $X$. For $1\le i\neq j\le m$, we have $$V_j\subset O_j\subset (\bigcup_{i\neq j }K_i)^c,$$ and hence $K_i\cap V_j=\emptyset$. It follows that  $K_i\subset V_i$.
Set
$$\a=\{V_1,\ldots,V_i-\cup_{j=1}^{i-1}V_j,\ldots,V_m-\cup_{j=1}^{m-1}V_j\}.$$ By construction, we have $\a\in \{\a_l\}_{l=1}^\infty.$ It is easy to see that for each $i=1,\dots,m$,
$$K_i\subset V_i-\cup_{j=1}^{i-1}V_j\subset O_i.$$
Hence $$\mu(B_i\Delta (V_i-\cup_{j=1}^{i-1}V_j))\leq 2\mu(O_i-K_i)<\frac{\ep}{m}, $$
and then $\mu(\a\Delta\beta)<\ep.$
This completes the proof.
\end{proof}

The following lemma is inspired by \cite[Lemma 4.4]{hy}.
\begin{lem}\label{l3}
	Let $X$ be a compact metric space and $Y$ be a non-empty  subset of $X$. Consider $f\in C(X)$, $\U\in \mathcal{C}^o_X$ and $\{\a_l\}_{l=1}^K$ as $K$ finite  partitions of $X$ that are finer than $\U$. For each $N\in\N$ and $\ep>0$, there exists a finite subset $B_N$ of $Y$ such that for each $l=1,\ldots,K$, each atom of $\a_l$ contains at most one point of $B_N$, and $$\sum_{y\in B_N}\mathrm{e}^{f(y)}\geq\frac{\inf\{\sum_{V\in \V}\sup_{y\in V}\emph{e}^{f(y)}:\V\in \mathcal{C}_{Y},\V\succeq\U\}}{K}-\ep.$$
	In particular, one can choose $B_N$ such that
$$\sum_{y\in B_N}\mathrm{e}^{f(y)}\geq\frac{1}{2K}\inf\{\sum_{V\in \V}\sup_{y\in V}\mathrm{e}^{f(y)}:\V\in \mathcal{C}_{Y},\V\succeq\U\}.$$
\end{lem}	
\begin{proof}
Given $N\in \N$ and $\ep>0$.
	Choose $y_1\in Y$ such that
$$\text{e}^{f(y_1)}>\sup\limits_{y\in Y} \text{e}^{f(y)}-\frac{\ep}{2}.$$
Let
$$Y_1=Y-\bigcup\limits_{l=1}^K\a_l(y_1),$$
where $\a_l(y_1)$ denotes the atom of $\a_l$  that contains $y_1.$ If $Y_1=\emptyset,$ set $B_N=\{y_1\}$. Otherwise choose $y_2\in Y_1$ such that
$$\text{e}^{f(y_2)}>\sup\limits_{y\in Y_1} \text{e}^{f(y)}-\frac{\ep}{2^2}.$$
 Let
 $$Y_2=Y_1-\bigcup\limits_{l=1}^K\a_l(y_2).$$
  If $Y_2=\emptyset,$ set $B_N=\{y_1,y_2\}$. Otherwise choose $y_3\in Y_2$ with
  $$\text{e}^{f(y_3)}>\sup\limits_{y\in Y_2} \text{e}^{f(y)}-\frac{\ep}{2^3}.$$
  Since each $\a_l$ is a finite partition, continuing inductively the procedure,  we get a sequence of sets  $Y_1,\ldots,Y_m\subset Y$ and a finite set $B_N=\{y_1,\ldots,y_m\}\subset Y$.  It is easy to see that for each $l=1,\ldots,K$, each atom of $\a_l$ contains at most one point of $B_N$. Now let
  $$\beta=\{Y_{j-1}\cap \a_l(y_j)\}_{1\leq j\leq m,1\leq l\leq K}$$
  with $Y_0=Y$.
It follows that $\beta\in \mathcal{C}_Y$ and $\beta\succeq\U$. Hence
	\begin{equation*}
		\begin{split}
			\sum_{y\in B_N}\text{e}^{f(y)}=\sum_{j=1}^m\text{e}^{f(y_j)}&\geq \sum_{j=1}^m\frac{1}{K}\sum_{l=1}^K (\sup_{y\in Y_{j-1}\cap \a_l(y_j)}\text{e}^{f(y)}-\frac{\ep}{2^j})\\
			&\geq\frac{\inf\{\sum_{V\in \V}\sup_{y\in V}\text{e}^{f(y)}:\V\in \mathcal{C}_{Y},\V\succeq\U\}}{K}-\ep.
		\end{split}
	\end{equation*}
This completes the proof.
\end{proof}	

We now prove the invertible case.
\begin{customthm}{{\bf Theorem \ref{thm:measure-relation}}  (The invertible case)}\label{thm:measure-relation:invertible}
{\it Let $(Y,S)$ be an invertible TDS, $\mathbfcal{F}=\{{\bf f}_n\}_{n=1}^\infty$ be  a block sub-additive potential, $\mu\in M(Y,S)$ and $\U\in \C_Y^o$. Then for the sub-additive potential $\mathcal{F}$ induced by $\mathbfcal{F}$, we have
 $$ GP_\mu (S,\mathbfcal{F},\U)\leq P_\mu(S,\mathcal{F}, \U)= h_{\mu}(S,\U)+\mathcal{F}_*(\mu).$$
If additionally $\mu$ is ergodic, then
$$ GP_\mu (S,\mathbfcal{F},\U)=P_\mu(S,\mathcal{F}, \U)= h_{\mu}(S,\U)+\mathcal{F}_*(\mu).$$}
\end{customthm}

\begin{proof}
We proceed with the proof in two steps.

\medskip
{\bf Step 1}:  Prove that $ GP_\mu (S,\mathbfcal{F},\U)\leq h_{\mu}(S,\U)+ \mathcal{F}_*(\mu).$

 \smallskip
 The upper bound is immediate if $GP_\mu(S,\mathbfcal F,\U)=-\infty$. In the rest of this step, we assume that $GP_\mu(S,\mathbfcal F,\U)>-\infty$.

 Let $\{g_k\}_{k=1}^\infty$ be a dense subset of $C(Y)$. First we  show that:

\medskip
{\bf Claim 1}:
 There exist suitable sequences $\{\delta_n\}$, tending to $0$,  and $\{L_n\}_{n=1}^\infty\subset \L_{C(Y)}$ such that for any $g\in \{g_k\}_{k=1}^\infty$, there exists $N\in\N$ such that $g\in L_n$ whenever $n>N$, and
\begin{equation*}
	\begin{split}		
GP_\mu (S,\mathbfcal{F},\U)\leq\varlimsup\limits_{n\to\infty}\frac{1}{n}\log GP_{n,\delta_n,\mu,L_n}(S,\mathbfcal{F},\U). 	
	\end{split}
\end{equation*}

\smallskip
 {\it Proof of Claim 1}:
 For any $k\in\N,$ by definition we have
   $$\varlimsup\limits_{n\to\infty}\frac{1}{n}\log GP_{n,\frac{1}{k},\mu,\{g_1,\ldots,g_k\}}(S,\mathbfcal{F},\U)> GP_\mu(S,\mathbfcal{F},\U)-\frac{1}{k}.$$
 Choose an increasing sequence $\{n_k\}_{k=1}^\infty$ such that
 $$\frac{1}{n_k}\log GP_{n_k,\frac{1}{k},\mu,\{g_1,\ldots,g_k\}}(S,\mathbfcal{F},\U)> GP_\mu(S,\mathbfcal{F},\U)-\frac{1}{k},\ \forall\  k\in\N.$$	
Define $\delta_n$ and $L_n$ as follows: if $n\notin\{n_k\}_{k=1}^\infty$, set
 $$\delta_n=\frac{1}{n},\ \text{and }  L_{n}=\{g_1,\ldots,g_n\};$$ otherwise, for each  $n\in\{n_k\}_{k=1}^\infty$, set $$\delta_{n_k}=\frac{1}{k}, \ \text{and }  L_{n_k}=\{g_1,\ldots,g_k\},\  k\in\N.$$
 It is easy to check that $\{\delta_n\}_{n=1}^\infty$ and $\{L_n\}_{n=1}^\infty$ satisfy the required property.
This completes the proof of the Claim 1.

\medskip
Put $\mathcal N=\{n\ge2:Y^n_{\delta_n,\mu,L_n}\ne\emptyset\}.$ The choice of $\{n_k\}_{k=1}^\infty$ implies that $\mathcal N$ is infinite. For $n\ge 2$ outside $\mathcal N$, we have $GP_{n,\delta_n,\mu,L_n}(S,\mathbfcal{F},\U)=0$. Consequently,
$$
\limsup_{n\to\infty}\frac1n
\log GP_{n,\delta_n,\mu,L_n}(S,\mathbfcal{F},\U)
=
\limsup_{\substack{n\to\infty\\ n\in\mathcal N}}\frac1n
\log GP_{n,\delta_n,\mu,L_n}(S,\mathbfcal{F},\U).
$$
In the remainder of Step 1, the measures below are constructed only for $n\in\mathcal N$, and all limits involving these measures are taken along $\mathcal N$.

By Lemma \ref{count},
we can find a countable sequence $\{\a_l\}_{l=1}^\infty\subset \U^*$ such that $\mu(\partial\a_l)=0$ for each $l\in\N$, and \begin{equation}\label{equation-4}
h_{\mu}^+(S,\U)=\inf_{l\in\N} h_{\mu}(S,\a_l).
\end{equation}
For  $n\in\mathcal N$, since  $\{(\a_l)^n\}_{l=1}^n$ are $n$ finite partitions of $Y^n$  that are finer than $\U^n$, then by Lemma \ref{l3}, there exists a finite subset $B_n\subset Y^n_{\delta_n,\mu,L_n}$ such that for each $l$, each atom of $(\a_l)^n$ contains at most one point of $B_n$, and
\begin{equation}\label{equation-5}
\sum_{{\bf x}\in B_n}\text{e}^{{\bf f}_n({\bf x})}\geq\frac{1}{2n}GP_{n,\delta_n,\mu,L_n}(S,\mathbfcal{F},\U).
\end{equation}
Set
$$\mu_n=\frac{1}{n}\sum_{{\bf x}\in B_n}\sum_{j=0}^{n-1}b_{n, {\bf x}}\delta_{x_j}=\sum_{{\bf x}\in B_n}b_{n, {\bf x}}\Bigg(\frac{1}{n}\sum_{j=0}^{n-1}\delta_{x_j}\Bigg),$$
where  $$b_{n, {\bf x}}=\frac{\text{e}^{{\bf f}_n({\bf x})}}{\sum_{{\bf y}\in B_n}\text{e}^{{\bf f}_n({\bf y})}}.$$

\medskip
{\bf Claim 2}: We have that $\mu_n\to \mu$ as $n\to \infty$.

\smallskip
{\it Proof of Claim 2}: Let $\ep>0$. For each $f\in C(Y)$,  there exists some $g\in \{g_k\}_{k=1}^\infty$ such that $||f-g||_\infty<\frac{\ep}{3}$. By Claim 1, we can find a  large enough $N\in\N$ such that for any $n\ge N$, $\delta_n\in (0,\frac{\ep}{3})$ and $g\in L_n$. Observe that $B_n\subset Y^n_{\delta_n,\mu,L_n}$ and $\sum_{{\bf x}\in B_n}b_{n, {\bf x}}=1$. Then
\begin{align*}
\bigg|\int f \text{d}\mu_n-\int f\text{d}\mu\bigg| &\le  \Bigg|\sum_{{\bf x}\in B_n}b_{n, {\bf x}}\Bigg(\frac{1}{n}\sum_{j=0}^{n-1}f({x_j})\Bigg)
-\sum_{{\bf x}\in B_n}b_{n, {\bf x}}\Bigg(\frac{1}{n}\sum_{j=0}^{n-1}g({x_j})\Bigg)\Bigg|\\
&\ \ \ \ +\Bigg|\sum_{{\bf x}\in B_n}b_{n, {\bf x}}\Bigg(\frac{1}{n}\sum_{j=0}^{n-1}g({x_j})\Bigg)-\int g\text{d}\mu\Bigg|
+\bigg|\int g \text{d}\mu-\int f\text{d}\mu\bigg|\\
&\le \frac{\ep}{3}+\delta_n+\frac{\ep}{3}< \ep.
\end{align*}
Since $f\in C(Y)$ is arbitrary, $\mu_n\to\mu$ in the weak$^*$ topology. This completes the proof of Claim 2.

\medskip
Fix $x^*\in Y$.
Set
$$\tilde{\eta}_n=\sum_{{\bf x}\in B_n}b_{n, {\bf x}}\delta_{({\bf x}, x^*, x^*,\dots)}.$$
It is clear that $\tilde{\eta}_n\in M(Y^\N).$
Put $$\tilde{\nu}_n=\frac{1}{n}\sum_{j=0}^{n-1}\tilde{\eta}_n\circ \sigma^{-j}.$$	
Define the orbit embedding $\phi:Y\to Y^\N$ by
$$\phi(x)=(x, Sx,\ldots,S^{n}x, \ldots).$$
It induces naturally an embedding map  $\phi_*: M(Y)\to M(Y^{\N})$.

\smallskip
{\bf Claim 3}:   We have that $\tilde{\nu}_n\to\phi_*(\mu)$ as $n\to\infty$.

\smallskip
{\it Proof of Claim 3}:  When $n\to\infty$,   we have $\mu_n\to \mu$ by Claim 2,  and then
$$\phi_*(\mu_n)=\frac{1}{n}\sum_{{\bf x}\in B_n} \sum_{j=0}^{n-1} b_{n,{\bf x}}\delta_{(x_j,Sx_j,S^2x_j,\ldots)}\to \phi_*(\mu).$$
Take any $\eta>0$. Choose some $2\le N\in \N$ such that $\frac{\diam(Y)}{2^{N-2}}<\eta.$
By continuity, there exists some $0<\theta<\frac{\eta}{2N}$ such that whenever $x,y\in Y$ with
$d(x,y)<\theta$, we have
$$
d(S^ix,S^iy)<\frac{\eta}{2N}, \ i=0,\ldots,N-1.
$$
Since $\delta_n\to 0$, we can choose $N_0>{N}/{\eta}$ such that for all $n>N_0$,
$\delta_n<\theta^2$. Now fix an $n>N_0$. For each ${\bf x}\in B_n$, denote
$$
B_{n,{\bf x}}=\{0\le j\le n-2: d(Sx_j, x_{j+1})>\sqrt{\delta_n}\}.
$$
We decompose $\{j:0\leq j\leq n-1\}$ into disjoint three subsets as follows:
$$
 S_1=\{0\leq j\leq n-N: j,j+1,\ldots,j+N-2\notin B_{n,{\bf x}}\},
 $$
 $$
 S_2=\big\{0\leq j\leq n-N: \{j,j+1,\ldots,j+N-2\}\cap B_{n,{\bf x}}\neq\emptyset\big\},
 $$
 and
$$
S_3= \{j: n-N+1\leq j\leq n-1\}.
$$
 When $j\in  S_1$, one has
 $$d(Sx_s,x_{s+1})\leq\sqrt{\delta_n},\  s=j, j+1,\ldots, j+N-2.$$
It follows that  for every $t=0,\ldots,N-1$,  $$d(S^tx_j,x_{j+t})\leq\sum_{\ell=0}^{t-1}d(S^{t-\ell}x_{j+\ell},S^{t-\ell-1}x_{j+\ell+1})<\frac{t\eta}{2N}<\frac{\eta}{2}.$$
Then
$$
 \widetilde d\big((x_j,x_{j+1},\ldots),(x_j,Sx_j,S^2x_j,\ldots)\big) \le\frac{\eta}{2N}\sum_{t=0}^{N-1}\frac{t}{2^t}
  +\frac{\operatorname{diam}(Y)}{2^{N-1}}<\eta.
$$
and so
$$D\big(\delta_{(x_j,x_{j+1},x_{j+2},\ldots)},\delta_{(x_j,Sx_j,S^2x_j,\ldots)}\big)\leq\eta.$$
When $j\in S_2\cup S_3$, it is easy to see that
$$
D\big(\delta_{(x_j,x_{j+1},\ldots)},
\delta_{(x_j,Sx_j,S^2x_j,\ldots)}\big)
\le 2\operatorname{diam}(Y).
$$
Since  ${\bf x}\in B_n\subset Y^n_{\delta_n},$ then by Remark \ref{rem:average-pseudo-orbit} (1), $|B_{n,{\bf x}}|<(n-1)\sqrt{\delta_n}.$ Consequently,
$$|S_2|\leq N|B_{n,{\bf x}}|< Nn\sqrt{\delta_n}\leq \frac{n\eta}{2}.$$
Moreover, we have $$|S_3|=N-1<n\eta.$$
It follows that
\begin{equation*}
\begin{split}
D\big(\widetilde\nu_n,\phi_*\mu_n\big) &\le\frac1n\sum_{{\bf x}\in B_n}b_{n,{\bf x}}
 \big(n\eta+2\operatorname{diam}(Y)(|S_2|+|S_3|)\big)\\
  &\le\eta+3\eta\operatorname{diam}(Y),
\end{split}
\end{equation*}
where $S_2$ and $S_3$ are the sets associated with ${\bf x}$. Since $\eta>0$ is arbitrary, $D(\widetilde\nu_n,\phi_*\mu_n)\to0$. Together with $\phi_*\mu_n\to\phi_*\mu$, this completes the proof of Claim 3.

\medskip
Let $\pi_{n-1}:Y^\N\to Y^n$ be the  projection  onto the  initial $n$ components. For each $l=1,\dots, n$, denote $$\pi_0^{-1}\a_l=\{\pi_0^{-1}(A):A\in \a_l\}.$$ Then $\pi_0^{-1}\a_l\in\mathcal{P}_{Y^\N}.$

\medskip
{\bf Claim 4}:
Fix $l,p\in\N$ and $n\in\mathcal N$ with $n>\max\{l,p\}$. We have
\begin{align*}
\frac{1}{n}\log GP_{n,\delta_n,\mu,L_n}(S,\mathbfcal{F},\U)\le \frac{1}{p}H_{\tilde{\nu}_n}\Big(\bigvee\limits_{i=0}^{p-1}\sigma^{-i}(\pi_0^{-1}\a_l)\Big)+\frac{1}{n}	\int {\bf f}_n\circ\pi_{n-1}\text{d}\tilde{\eta}_n+\frac{2p\log m+\log 2n}{n}.	
\end{align*}
{\it Proof of Claim 4}: For fixed $l,n\in\N$ with $n>l$, since each atom of $(\a_l)^n$ contains at most one point of $B_n$, then combining with inequality \eqref{equation-5} and Lemma \ref{lw}, we have
\begin{equation*}
	\begin{split}
\log GP_{n,\delta_n,\mu,L_n}(S,\mathbfcal{F},\U)-\log 2n & \leq\log\sum_{{\bf x}\in B_n}\text{e}^{{\bf f}_n({\bf x})}\\
&=\sum_{{\bf x}\in B_n} b_{n,{\bf x}}\Big({\bf f}_n({\bf x})-\log b_{n,{\bf x}}\Big)	\\
&=\sum_{{\bf x}\in B_n} \tilde{\eta}_n(\{({\bf x},x^*, \dots)\})({\bf f}_n({\bf x})-\log \tilde{\eta}_n(\{({\bf x},x^*,\dots)\}))\\
&=H_{\tilde{\eta}_n}\Big(\bigvee\limits_{i=0}^{n-1}\sigma^{-i}(\pi_0^{-1}\a_l)\Big)	+\int {\bf f}_n\circ\pi_{n-1}\text{d}\tilde{\eta}_n.
	\end{split}
\end{equation*}
Fix $p \in \N$ with  $p<n$. For each $j=0,1,\ldots,p-1,$ rewrite
$$\bigvee_{i=0}^{n-1}\sigma^{-i}(\pi_0^{-1}\a_l)=\bigvee_{r=0}^{[\frac{n-j}{p}]-1}\sigma^{-(pr+j)}\Bigg(\bigvee\limits_{i=0}^{p-1}
\sigma^{-i}(\pi_0^{-1}\a_l)\Bigg)\vee\bigvee_{t\in S_j}\sigma^{-t}(\pi_0^{-1}\a_l),$$
where $S_j=\{0,1,\ldots,j-1\}\cup\{j+p[\frac{n-j}{p}],\ldots,n-1\}.$
Then
\begin{equation*}
\begin{split}
	&H_{\tilde{\eta}_n}\Big(\bigvee\limits_{i=0}^{n-1}\sigma^{-i}(\pi_0^{-1}\a_l)\Big)	+\int {\bf f}_n\circ\pi_{n-1}\text{d}\tilde{\eta}_n\\
	\leq&	\sum_{r=0}^{[\frac{n-j}{p}]-1}H_{\sigma^{(pr+j)}\tilde{\eta}_n}	\Big(\bigvee\limits_{i=0}^{p-1}\sigma^{-i}(\pi_0^{-1}\a_l)\Big)+2p\log m+\int {\bf f}_n\circ\pi_{n-1}\text{d}\tilde{\eta}_n.
\end{split}
\end{equation*}
Summing  over $j$ from $0$ to $p-1$ and dividing by $pn$, we have
\begin{equation*}
	\begin{split}
		&	\frac{1}{n}\log GP_{n,\delta_n,\mu,L_n}(S,\mathbfcal{F},\U)\\
	\leq&\ \frac{1}{pn}\sum_{j=0}^{n-1}H_{\sigma^j\tilde{\eta}_n}\Big(\bigvee\limits_{i=0}^{p-1}\sigma^{-i}(\pi_0^{-1}\a_l)\Big)+\frac{1}{n}	\int {\bf f}_n\circ\pi_{n-1}\text{d}\tilde{\eta}_n+\frac{2p\log m+\log 2n}{n}\\	
	\leq&\ \frac{1}{p}H_{\tilde{\nu}_n}\Big(\bigvee\limits_{i=0}^{p-1}\sigma^{-i}(\pi_0^{-1}\a_l)\Big)+\frac{1}{n}	\int {\bf f}_n\circ\pi_{n-1}\text{d}\tilde{\eta}_n+\frac{2p\log m+\log 2n}{n}.	
	\end{split}
\end{equation*}
The last inequality holds because $H_{\{\cdot\}}\Big(\bigvee\limits_{i=0}^{p-1}\sigma^{-i}(\pi_0^{-1}\a_l)\Big)$ is concave. This completes the proof of Claim 4.

\medskip
Since $\mathbfcal{F}=\{{\bf f}_n\}_{n=1}^\infty$ is a block sub-additive potential on $Y$, it is not hard to verify that $\{{\bf f}_n\circ \pi_{n-1}\}_{n=1}^\infty$ forms a sub-additive potential for $(Y^\N,\sigma)$.
By Lemma \ref{subadd}, we have $$\varlimsup_{n\to \infty}\frac{1}{n}	\int {\bf f}_n\circ\pi_{n-1}\text{d}\tilde{\eta}_n\leq\lim_{n\to \infty}\frac{1}{n}\int {\bf f}_n\circ\pi_{n-1} \text{d}\phi_*(\mu).$$
Observe that $\partial(\pi_0^{-1}\a_l)\subset \pi_0^{-1}(\partial\a_l)$ and $\mu(\partial\a_l)=0$, then
 $\phi_*(\mu)(\partial(\pi_0^{-1}\a_l))=0.$ Hence
\begin{align*}
	&\ \ \ \varlimsup\limits_{n\to\infty}\frac{1}{n}\log GP_{n,\delta_n,\mu,L_n}(S,\mathbfcal{F},\U)\\
&\leq \frac{1}{p}H_{\phi_*(\mu)}\Big(\bigvee\limits_{i=0}^{p-1}\sigma^{-i}(\pi_0^{-1}\a_l)\Big)+\lim_{n\to\infty}\frac{1}{n}\int {\bf f}_n\circ\pi_{n-1} \text{d}\phi_*(\mu)\\
&=\frac{1}{p}H_{\mu}\Big(\bigvee\limits_{i=0}^{p-1}S^{-i}{\a}_l\Big)+\lim_{n\to\infty}\frac{1}{n}\int{{f_n}}\ \text{d}\mu.	
\end{align*}
Let $p\to \infty,$ we obtain that
$$GP_\mu(S,\mathbfcal{F},\U)\leq h_\mu(S,\a_l)+ \mathcal{F}_*(\mu).$$
Furthermore, by Lemma \ref{lb} and equation \eqref{equation-4}, we have
$$h_{\mu}(S,\U)=h_\mu^+(S,\U)=\inf_{l\in\N}h_\mu(S,\a_l).$$
 Hence
$$ GP_\mu (S,\mathbfcal{F},\U)\leq h_{\mu}(S,\U)+ \mathcal{F}_*(\mu).$$

\medskip
{\bf Step 2}:  Prove that
$
GP_\mu (S,\mathbfcal{F},\U)= h_{\mu}(S,\U)+ \mathcal{F}_*(\mu)
$
when $\mu$ is ergodic.

\smallskip
Denote $$\text{orb}_n(Y)=\{(y,Sy,\ldots,S^{n-1}y):y\in Y\}.$$
Fix $\delta>0$ and $L\in\L_{C(Y)}$. Put
$$Y_n=\Big\{y\in Y:\max\limits_{g\in L}\Big|\frac{1}{n}\sum_{i=0}^{n-1}g(S^iy)-\int g \text{d}\mu\Big|\leq\delta\Big\}.$$
As $\mu$ is ergodic, according to the Birkhoff pointwise ergodic theorem, we have  $\mu(Y_n)\to 1$ when $n\to\infty$.
Observe that
\begin{equation*}
	\begin{split}
&\varlimsup\limits_{n\to\infty}\frac{1}{n}\log\inf\{\sum_{V\in \V}\sup_{{\bf y}\in V}\text{e}^{{\bf f}_n({\bf y})}:\V\in \mathcal{C}_{Y^n_{\delta,\mu,L}},\V\succeq\U^n\}\\
\ge  &\varlimsup\limits_{n\to\infty}\frac{1}{n}\log \inf\{\sum_{V\in \V}\sup_{{\bf y}\in V}\text{e}^{{\bf f}_n({\bf y})}:\V\in \mathcal{C}_{\text{orb}_n(Y)\cap Y^n_L(\delta,\mu)},\V\succeq\U^n\}\\
=&\ \varlimsup\limits_{n\to\infty}\frac{1}{n}\log \inf\{\sum_{V\in \V}\sup_{y\in V}\text{e}^{f_n(y)}:\V\in \mathcal{C}_{Y_n},\V\succeq\U^{n-1}_0\}.		
	\end{split}
\end{equation*}
To finish the proof, it suffices to show that
\begin{equation*}
		\begin{split}
&\varlimsup\limits_{n\to\infty}\frac{1}{n}\log \inf\{\sum_{V\in \V}\sup_{y\in V}\text{e}^{f_n(y)}:\V\in \mathcal{C}_{Y_n},\V\succeq\U^{n-1}_0\}
\geq \ h_\mu(S,\U)+		
\lim_{n\to \infty}\int_{Y}\frac{f_n}{n}\text{d}\mu.		
	\end{split}
\end{equation*}
Based on Remark \ref{rem:measure-pressure-GSP} (1), we only need to estimate
$\sum\limits_{B\in \beta}\sup\limits_{y\in B}\text{e}^{f_n(y)}$ for a given $\beta\in  \mathcal{P}_{Y_n} $ with $\beta\succeq\U^{n-1}_0.$

We take $n$ sufficiently large that $\mu(Y_n)>0$ and use the convention $0\log 0=0$. When $\mu(Y_n^c)=0$, the entropy contribution from the sets $A_i$ below is zero, and the normalization by $\mu(Y_n^c)$ is used only when $\mu(Y_n^c)>0$. By Lemma \ref{lw}, we have
\begin{align*}
		\log\sum_{B\in \beta}\sup_{y\in B}\text{e}^{f_n(y)}&\geq \sum_{B\in \beta}\frac{\mu(B)}{\mu(Y_n)}\Big(\sup_{y\in B}f_n(y)-\log\frac{\mu(B)}{\mu(Y_n)}\Big)\\
		&\geq\frac{1}{\mu(Y_n)}\Big(\int_{Y_n}f_n\ \text{d}\mu-\sum_{B\in \beta}\mu(B)\log\mu(B)\Big)+\log\mu(Y_n).
\end{align*}
Choose a partition $\a=\{A_1,\ldots,A_N, Y_n\}\in \mathcal{P}_Y$ such that each $A_i \ (i=1,\ldots,N)$ is contained in some element of $\U^{n-1}_0$,  and $N\leq N(\U^{n-1}_0).$ Denote $\tilde{\beta}=\beta\cup \{Y_n^c\}$. Then $\tilde{\beta}\in \mathcal{P}_Y$ and $\a\vee \tilde{\beta}\succeq\U^{n-1}_0.$ It follows that $H_\mu(\a\vee \tilde{\beta})\geq H_\mu(\U^{n-1}_0).$ On the other hand,  we have
\begin{equation*}
	\begin{split}
	H_\mu(\a\vee \tilde{\beta})&=-\sum_{B\in \beta}\mu(B)\log\mu(B)-\sum_{i=1}^N	\mu(A_i)\log\mu(A_i)\\
	&\leq -\sum_{B\in \beta}\mu(B)\log\mu(B)+\mu(Y_n^c)\big(\log N-\log \mu(Y_n^c)\big).\\
	\end{split}
\end{equation*}
The final inequality can be derived from Lemma \ref{lw} by considering the specific case where $a_i=0$ and $p_i=\mu(A_i)/\mu(Y_n^c)$.
Hence
\begin{equation*}
	\begin{split}
\frac{1}{n}\log\sum_{B\in \beta}\sup_{y\in B}\text{e}^{f_n(y)}  \ge & \frac{1}{\mu(Y_n)}\Big(\int_{Y_n}\frac{f_n}{n}\text{d}\mu+\frac{1}{n}H_\mu(\U^{n-1}_0)\\
&\ -\mu(Y_n^c)\cdot\frac{\log N(\U^{n-1}_0)-\log \mu(Y_n^c)}{n}\Big)+\frac{\log\mu(Y_n)}{n}.
	\end{split}
\end{equation*}
Note that by sub-additivity,
$$\sup_{y\in Y}f_n(y)\leq n \sup_{y\in Y}{f}_1(y),$$ and then
\begin{equation*}
	\begin{split}
\int_{Y_n}\frac{f_n}{n}\text{d}\mu& =\int_{Y}\frac{f_n}{n}\text{d}\mu-\int_{Y_n^c}\frac{f_n}{n}\text{d}\mu\\
&\geq 	\int_{Y}\frac{f_n}{n}\text{d}\mu-\mu(Y_n^c)\sup_{y\in Y}{f}_1(y).	
	\end{split}
\end{equation*}
Taking infimum over $\beta$ and letting $n\to\infty$, it follows that
\begin{equation*}
	\begin{split}
		\varlimsup\limits_{n\to\infty}\frac{1}{n}\log \inf\{\sum_{V\in \V}\sup_{y\in V}\text{e}^{f_n(y)}:\V\in \mathcal{C}_{Y_n},\V\succeq\U^{n-1}_0\}
		\geq\ h_\mu(S,\U)+		
		\lim_{n\to \infty}\int_{Y}\frac{f_n}{n}\text{d}\mu.		
	\end{split}
\end{equation*}
This completes the proof.
\end{proof}

\subsection{Proof of Theorem \ref{thm:measure-relation}: the non-invertible case}

We now extend the preceding result to the general case by working in a two-sided full shift.

Let $(X,T)$ and $(Y,S)$ be two TDSs. A continuous surjection $p:X\to Y$ is called a {\it factor map} if $p\circ T=S\circ p$.
For $\nu\in M(X)$, write $p_*\nu=\nu\circ p^{-1}$.

\begin{lem}[\cite{R03}, Proposition 6]\label{ld}
Let $p:(X,T)\to(Y,S)$ be a factor map and $\U\in\mathcal C_Y^o$. Then for every $\nu\in M(X,T)$, we have
$$
h_\nu(T,p^{-1}\U)=h_{p_*\nu}(S,\U).
$$
\end{lem}

Let $Z=Y^{\mathbb Z}$ and  $\sigma:Z\to Z$ be the left shift. Equip $Z$ with a compatible product metric, and write $D_Z$ for the corresponding weak$^*$ metric on $M(Z)$ defined in subsection \ref{subsec:induced-spaces}. Set
$$
X=\{z=(z_j)_{j\in\mathbb Z}\in Z: Sz_j=z_{j+1}\text{ for every }j\in\mathbb Z\}
$$
and $T=\sigma|_X$. Then $(Z,\sigma)$ and $(X,T)$ are invertible TDSs. Let $\pi:Z\to Y$ be given by $\pi(z)=z_0$, and put $p=\pi|_X$.
Since $S$ is surjective, $p:(X,T)\to(Y,S)$ is a factor map. We identify measures on $X$ with measures on $Z$ supported on $X$.

Fix $y_*\in Y$. For $n\ge 2$, define $\iota_n:Y^n\to Z$ by 
$$
\big(\iota_n({\bf y})\big)_j=\begin{cases}
y_j, & 0\le j\le n-1,\\
y_*, &   j<0\text{ or }j\ge n
\end{cases}
$$
for every ${\bf y}=(y_0,\ldots,y_{n-1})$, and put
$$
\lambda_n({\bf y}) =\frac1n\sum_{j=0}^{n-1}\delta_{\sigma^j\iota_n({\bf y})}.
$$
In the following two lemmas, we use the equivalent definition of measure-theoretic pressure in Remark~\ref{rem:measure-pressure-GSP}(2).

\begin{lem}\label{l5}
Let $\mu\in M(Y,S)$. There exists a unique $\nu\in M(X,T)$ with $p_*\nu=\mu$. Moreover, for every $\varepsilon>0$, there exist $\delta>0$ and $N\ge2$ such that whenever $n\ge N$ and ${\bf y}\in Y^n_{\delta,\mu}$, we have
$$
D_Z\big(\lambda_n({\bf y}),\nu\big)<\varepsilon.
$$
\end{lem}

\begin{proof}
Define $e\in C(Z)$ by $e(z)=d(Sz_0,z_1)$. To prove existence, define a continuous map $\iota:Y\to Z$ by
$$
\big(\iota(y)\big)_j=
\begin{cases}
S^j y,& j\ge 0,\\
y_*,& j<0.
\end{cases}
$$
Any weak$^*$ limit point $\nu$ of the averages $\frac1N\sum_{j=0}^{N-1}(\sigma^j)_*(\iota_*\mu)$
is $\sigma$-invariant and satisfies $\pi_*\nu=\mu$ and $\int e\,\mathrm d\nu=0$. Since $e\ge 0$, $\int e\, \mathrm d\nu=0$, and $\nu$ is $\sigma$-invariant, we have $e\circ\sigma^j=0$ $\nu$-a.e. for every $j\in\mathbb Z$. Hence
$$
\nu(X)=\nu\left(\bigcap_{j\in\mathbb Z}\{e\circ\sigma^j=0\}\right)=1.
$$
Thus $\nu\in M(X,T)$ and $p_*\nu=\mu$.

For uniqueness, let $\rho\in M(X,T)$ with $p_*\rho=\mu$. For any integers $a\le b$ and Borel sets $A_a,\ldots,A_b\subset Y$, since $T$ is invertible and $\rho$ is $T$-invariant,
\begin{align*}
\rho\big(\{z\in X:z_j\in A_j,\ a\le j\le b\}\big)
&=\rho\big(\{z\in X:z_j\in A_{a+j},\ 0\le j\le b-a\}\big)\\
&=\rho\left(p^{-1}\left(\bigcap_{j=0}^{b-a}S^{-j}A_{a+j}\right)\right)\\
&=\mu\left(\bigcap_{j=0}^{b-a}S^{-j}A_{a+j}\right).
\end{align*}
Here the second equality follows from $z_j=S^jz_0$ for $j\ge 0$ and $z\in X$. Thus any two such measures agree on all cylinder sets, and hence on the Borel $\sigma$-algebra of $X$. Therefore $\nu$ is unique.

Suppose that the last assertion is false. Then there exist $\varepsilon_0>0$, integers $n_k\to\infty$, and ${\bf y}^{(k)}\in Y^{n_k}_{1/k,\mu}$ such that $D_Z\big(\lambda_{n_k}({\bf y}^{(k)}),\nu\big)\ge\varepsilon_0$.
Passing to a subsequence, write $\lambda_{n_k}({\bf y}^{(k)})\to\lambda\in M(Z)$. For every $F\in C(Z)$, we have 
$$
\left|\int(F\circ\sigma-F)\,\mathrm d\lambda_{n_k}({\bf y}^{(k)})\right| \le\frac{2\|F\|_\infty}{n_k},
$$
and hence $\lambda$ is $\sigma$-invariant. Moreover,
$$
\pi_*\lambda_{n_k}({\bf y}^{(k)}) =\frac1{n_k}\sum_{j=0}^{n_k-1}\delta_{y_j^{(k)}}\to\mu,
$$
while the average pseudo-orbit condition yields 
$$
\int e\,\mathrm d\lambda_{n_k}({\bf y}^{(k)})
\le\frac1k+\frac{\operatorname{diam}(Y)}{n_k}\to0.
$$
Thus $\pi_*\lambda=\mu$. Since $\lambda$ is $\sigma$-invariant, $e\ge 0$, and $\int e\,\mathrm d\lambda=0$, we have $\lambda(X)=1$. Hence $\lambda\in M(X,T)$ and $p_*\lambda=\mu$. By uniqueness, $\lambda=\nu$. Since $D_Z$ is compatible with the weak$^*$ topology, this contradicts $D_Z\big(\lambda_{n_k}({\bf y}^{(k)}),\nu\big)\ge\varepsilon_0.$ This proves the lemma.
\end{proof}

For a block sub-additive potential $\mathbfcal F=\{{\bf f}_n\}_{n=1}^\infty$ on $Y$, define a block sub-additive potential $\mathbfcal F_\pi$ on $Z$ by
$$
\mathbfcal F_\pi =\{{\bf f}_n\circ\pi^{(n)}\}_{n=1}^\infty,
\quad \text{ with } \quad
\pi^{(n)}(w_0,\ldots,w_{n-1}) =(\pi(w_0),\ldots,\pi(w_{n-1})).
$$
Here only continuity of $\pi$ is needed, and the factor relation holds for its restriction $p$ to $X$.

\begin{lem}\label{lem:lift}
Let $\mu\in M(Y,S)$, $\U\in\mathcal C_Y^o$, and let $\nu$ be
the measure in Lemma~\ref{l5}. Then
$$
GP_\mu(S,\mathbfcal F,\U)
\le GP_\nu(\sigma,\mathbfcal F_\pi,\pi^{-1}\U).
$$
\end{lem}

\begin{proof}
Fix $\varepsilon>0$, and choose $\delta>0$ and $N$ as in Lemma~\ref{l5}. For $n\ge N$, define the continuous map $\Psi_n:Y^n\to Z^n$ by
$$
\Psi_n({\bf y}) =\big(\iota_n({\bf y}),\sigma\iota_n({\bf y}),\ldots, \sigma^{n-1}\iota_n({\bf y})\big)
$$
for ${\bf y}\in Y^n$. Since $\Psi_n({\bf y})$ is a genuine $\sigma$-orbit segment and its empirical measure is $\lambda_n({\bf y})$, by Lemma~\ref{l5} we obtain 
$$
\Psi_n\big(Y^n_{\delta,\mu}\big)\subset Z^n_{\varepsilon,\nu}.
$$
Moreover, it is clear that $\pi^{(n)}\circ\Psi_n=\operatorname{id}_{Y^n}$ and $({\bf f}_n\circ\pi^{(n)})\big(\Psi_n({\bf y})\big) ={\bf f}_n({\bf y})$.
Let $\mathcal V\in\mathcal C_{Z^n_{\varepsilon,\nu}}$ be finer than $(\pi^{-1}\U)^n$. For $V\in\mathcal V$, put $A_V=Y^n_{\delta,\mu}\cap\Psi_n^{-1}(V)$. The non-empty sets $A_V$ form a finite Borel cover of $Y^n_{\delta,\mu}$ finer than $\U^n$. Thus 
$$
\sum_{\substack{V\in\mathcal V\\ A_V\ne\emptyset}}
\sup_{{\bf y}\in A_V}\mathrm{e}^{{\bf f}_n({\bf y})}
\le
\sum_{V\in\mathcal V}
\sup_{{\bf w}\in V}
\mathrm{e}^{({\bf f}_n\circ\pi^{(n)})({\bf w})}.
$$
Taking the infimum over $\mathcal V$, we obtain that for all $n\ge N$,
$$
GP_{n,\delta,\mu}(S,\mathbfcal F,\U)
\le
GP_{n,\varepsilon,\nu}
(\sigma,\mathbfcal F_\pi,\pi^{-1}\U).
$$
The inequality is also immediate when
$Y^n_{\delta,\mu}=\emptyset$. Therefore,
$$
GP_\mu(S,\mathbfcal F,\U)
\le GP_{\delta,\mu}(S,\mathbfcal F,\U)
\le
GP_{\varepsilon,\nu}
(\sigma,\mathbfcal F_\pi,\pi^{-1}\U).
$$
Letting $\varepsilon\downarrow 0$ the lemma follows.
\end{proof}

We now prove the general case.

\begin{customthm}{{\bf Theorem \ref{thm:measure-relation}}(The non-invertible case)}
Let $(Y,S)$ be a TDS, $\mathbfcal F=\{{\bf f}_n\}_{n=1}^\infty$ be a block sub-additive potential, $\mu\in M(Y,S)$ and $\U\in\mathcal C_Y^o$.
Let $\mathcal F=\{f_n\}_{n=1}^\infty$ be the induced sub-additive potential. Then
$$
GP_\mu(S,\mathbfcal F,\U) \le P_\mu(S,\mathcal F,\U)=h_\mu(S,\U)+\mathcal F_*(\mu).
$$
If $\mu$ is ergodic, then
$$
GP_\mu(S,\mathbfcal F,\U) =P_\mu(S,\mathcal F,\U) =h_\mu(S,\U)+\mathcal F_*(\mu).
$$
\end{customthm}

\begin{proof}
The proof of Step 2 in \ref{thm:measure-relation:invertible} remains valid without invertibility. It therefore suffices
to prove the upper bound.

Let $\nu\in M(X,T)$ be the measure in Lemma~\ref{l5}, viewed as a measure in $M(Z,\sigma)$. By Lemma~\ref{lem:lift} and
Step 1 of \ref{thm:measure-relation:invertible}, applied to the invertible system $(Z,\sigma)$, we have
\begin{align*}
GP_\mu(S,\mathbfcal F,\U) &\le GP_\nu(\sigma,\mathbfcal F_\pi,\pi^{-1}\U)\\
&\le h_\nu(\sigma,\pi^{-1}\U) +\lim_{n\to\infty}\frac1n \int_Z {\bf f}_n(z_0,\ldots,z_{n-1})\,\mathrm d\nu(z).
\end{align*}
Since $\nu(X)=1$, restricting to $X$ and by Lemma~\ref{ld} we have
$$
h_\nu(\sigma,\pi^{-1}\U) =h_\nu(T,p^{-1}\U) =h_\mu(S,\U).
$$
Also, for $z\in X$, we have $z_j=S^j p(z)$ for all $j\ge 0$. Hence
$$
\int_Z {\bf f}_n(z_0,\ldots,z_{n-1})\,\mathrm d\nu(z) =\int_X f_n\circ p\,\mathrm d\nu =\int_Y f_n\,\mathrm d\mu.
$$
Combining these equalities gives
$$
GP_\mu(S,\mathbfcal F,\U) \le h_\mu(S,\U)+\mathcal F_*(\mu).
$$
This completes the proof.
\end{proof}

The following example shows that the equality in Theorem~\ref{thm:measure-relation} need not hold when $\mu$ is not ergodic.

\begin{exam}\label{ex:nonergodic-pressure}
There exist an invertible TDS $(X,T)$ with two points, a non-ergodic measure $\mu\in M(X,T)$ and a block sub-additive potential $\mathbfcal F=\{{\bf f}_n\}_{n=1}^\infty$ with $0\le {\bf f}_n\le n$ such that  for $\U=\{X\}$, we have 
$$
GP_\mu(T,\mathbfcal F,\U)=0,\quad h_\mu(T,\U)+\mathcal F_*(\mu)=1,
$$
where $\mathcal F$ is the induced potential. Moreover, the functions ${\bf f}_n$ can be written as logarithms of norms of products of invertible $2\times2$ matrices.
\end{exam}

\begin{proof}
Let $X=\{0,1\}$ with $d(0,1)=1$, $T=\operatorname{id}_X$, $\mu=\frac12(\delta_0+\delta_1)$ and $\U=\{X\}$. Put $s(0)=1$ and $s(1)=-1$, and define
$$
{\bf f}_n(x_0,\ldots,x_{n-1})=\left|\sum_{i=0}^{n-1}s(x_i)\right|.
$$
Each ${\bf f}_n$ is continuous and satisfies $0\le {\bf f}_n\le n$. By the triangle inequality,
$$
{\bf f}_{n+m}(x_0,\ldots,x_{n+m-1}) \le {\bf f}_n(x_0,\ldots,x_{n-1})+{\bf f}_m(x_n,\ldots,x_{n+m-1}),
$$
so $\mathbfcal F$ is block sub-additive. Since every orbit of $T$ is constant, the induced potential satisfies $f_n(x)=n$ for every $x\in X$. Hence $\mathcal F_*(\mu)=1$ and $h_\mu(T,\U)=0$.

Since $\U=\{X\}$, for every non-empty $A\subset X^n$,
$$
\inf_{\mathcal V\in\mathcal C_A} \sum_{V\in\mathcal V}\sup_{{\bf x}\in V}\mathrm{e}^{{\bf f}_n({\bf x})}
= \sup_{{\bf x}\in A}\mathrm{e}^{{\bf f}_n({\bf x})}.
$$
Indeed, $\{A\}$ is an admissible cover, while every cover of $A$ has total weight at least $\mathrm{e}^{{\bf f}_n({\bf x})}$ for each ${\bf x}\in A$.

Let $g\in C(X)$ be given by $g(x)=x$. If $n\ge2$, $\delta>0$ and ${\bf x}=(x_0,\ldots,x_{n-1})\in X^n_{\delta,\mu,\{g\}}$, then
$$
\frac1n{\bf f}_n({\bf x}) = \left|1-\frac2n\sum_{i=0}^{n-1}x_i\right| \le 2\delta.
$$
Therefore
$$
GP_{n,\delta,\mu,\{g\}}(T,\mathbfcal F,\U)\le \mathrm{e}^{2n\delta},
$$
where the inequality is also valid when $X^n_{\delta,\mu,\{g\}}=\emptyset$. It follows that
$$
GP_\mu(T,\mathbfcal F,\U) \le GP_{\delta,\mu,\{g\}}(T,\mathbfcal F,\U) \le 2\delta.
$$
Letting $\delta\downarrow 0$, we get $GP_\mu(T,\mathbfcal F,\U)\le 0$.

For the reverse inequality, fix $\delta>0$ and $L\in\mathcal{L}_{C(X)}$. For $n\ge2$, put
$$
{\bf w}_n=
(\underbrace{0,\ldots,0}_{\lfloor n/2\rfloor},
\underbrace{1,\ldots,1}_{\lceil n/2\rceil}).
$$
The word ${\bf w}_n$ has only one change of symbol, so its average pseudo-orbit error is $1/(n-1)$. Moreover, for every $h\in L$,
$$
\left|
\frac1n\sum_{i=0}^{n-1}h\big(({\bf w}_n)_i\big)
-\int h\,\mathrm d\mu
\right|
\le
\frac{|h(0)-h(1)|}{2n}.
$$
Since $L$ is finite, ${\bf w}_n\in X^n_{\delta,\mu,L}$ for all sufficiently large $n$. Thus $X^n_{\delta,\mu,L}\ne\emptyset$ for all sufficiently large $n$, and since ${\bf f}_n\ge 0$, $GP_{n,\delta,\mu,L}(T,\mathbfcal F,\U)\ge 1.$
Consequently, $GP_{\delta,\mu,L}(T,\mathbfcal F,\U)\ge 0$. Since $\delta>0$ and $L\in\mathcal{L}_{C(X)}$ are arbitrary, we obtain $GP_\mu(T,\mathbfcal F,\U)\ge 0$. Therefore
$$
GP_\mu(T,\mathbfcal F,\U)=0 < 1= h_\mu(T,\U)+\mathcal F_*(\mu).
$$

Finally, let
$$
A(0)=\begin{pmatrix}
\mathrm{e}& 0\\
0 & \mathrm{e}^{-1}
\end{pmatrix},
\quad
A(1)=\begin{pmatrix}
\mathrm{e}^{-1} & 0\\
0 & \mathrm{e}
\end{pmatrix}.
$$
Both matrices are invertible. For the operator norm induced by the Euclidean norm on $\mathbb R^2$, we have 
$$
\log\|A(x_{n-1})\cdots A(x_0)\| = \left|\sum_{i=0}^{n-1}s(x_i)\right| = {\bf f}_n(x_0,\ldots,x_{n-1}).
$$
This completes the proof.
\end{proof}

For comparison, let $\mathbfcal G=\{{\bf g}_n\}_{n=1}^\infty$, where ${\bf g}_n\equiv n$ on $X^n$. Then $\mathbfcal G$ is a block sub-additive potential inducing the same sub-additive potential $\mathcal F$. For every $\delta>0$ and $L\in\mathcal{L}_{C(X)}$, the sets $X^n_{\delta,\mu,L}$ are non-empty for all sufficiently large $n$, and every weight equals $\mathrm{e}^n$. Hence
$$
GP_\mu(T,\mathbfcal G,\U)= 1 \ne 0= GP_\mu(T,\mathbfcal F,\U).
$$
Thus different block sub-additive extensions of the same sub-additive potential may have different measure-theoretic pressures for a non-ergodic invariant measure.

\section{Topological proof of Theorem \ref{thm:topological-relation}}\label{Sect:proof-thm3}

We give a direct topological proof of Theorem~\ref{thm:topological-relation}.

\begin{customthm}{{\bf Theorem \ref{thm:topological-relation}}}
{\it	Let $(X,T)$ be a TDS, $\mathbfcal{F}=\{{\bf f}_n\}_{n=1}^\infty$ be a block sub-additive potential and $\U\in \C_X^o$. Then for the sub-additive potential $\mathcal{F}$ induced by $\mathbfcal{F}$, we have
	$$GP(T,\mathbfcal{F},\U)= P(T,\mathcal{F},\U).$$}
\end{customthm}

\begin{proof}
By definition,  we have $ P(T,\mathcal{F},\U)\leq GP (T,\mathbfcal{F},\U)$.
Next, we will show the opposite direction. 
For $c\in\mathbb R$, let $\mathbfcal{G}=\{{\bf g}_n\}_{n=1}^\infty$, where ${\bf g}_n={\bf f}_n+cn$, and let $\mathcal{G}=\{f_n+cn\}_{n=1}^\infty$ be its induced potential.
Then $\mathbfcal{G}$ is block sub-additive, and by definitions we have 
$$
P(T,\mathcal{G},\U)=P(T,\mathcal{F},\U)+c, \quad GP(T,\mathbfcal{G},\U)=GP(T,\mathbfcal{F},\U)+c.
$$
Thus, by choosing $c$ sufficiently large and retaining the same notation, we may assume that $P(T,\mathcal{F},\U)\ge 0$.
Fix $\ep>0.$ We only need to prove that $$ GP(T,\mathbfcal{F},\U)< P(T,\mathcal{F},\U)+\ep.$$

By definition,
	$$P(T,\mathcal{F},\U)=\lim\limits_{n\to\infty}\frac{1}{n}\log\inf\{\sum_{V\in \V}\sup_{x\in V}\text{e}^{f_n(x)}:\V\in \mathcal{C}_{X},\V\succeq\U_0^{n-1}\}.$$
Then for the fixed $\ep$,  there exists an $N\ge 2$, and $\V_0\in \mathcal{C}_{X}$ with $\V_0\succeq\U_0^{N-1}$,  such that
	$$\sum_{V\in \V_0}\sup_{x\in V}\text{e}^{{f}_N(x)}<\text{e}^{N(P(T,\mathcal{F},\U)+\frac{\ep}{2})}.$$
	Let $\phi_N:X\to X^N$ be defined as
$$\phi_N(x)=(x,Tx,\ldots,T^{N-1}x).$$
Then, the embedding map $\phi_N$  is a homeomorphism onto its image. It follows that
$\phi_N(\V_0)=\{\phi_N(V):V\in\V_0\}\in \C_{\phi_N(X)}$, $\phi_N(\V_0)\succeq\U^{N},$
  and
	$$\sum_{V\in \phi_N(\V_0)}\sup_{{\bf x}\in V}\text{e}^{{\bf f}_N({\bf x})}<\text{e}^{N(P(T,\mathcal{F},\U)+\frac{\ep}{2})}.$$
Since ${\bf f}_N$ is continuous on $X^N$,  by carefully choosing an open neighborhood  for each $V\in \phi_N(\V_0)$, we can derive a finite collection of open sets in $X^N$, denoted by $\V_1$, such that $\V_1\succeq\U^{N}$, $\V_1$ covers $\phi_N(X)$ and
\begin{equation}\label{equation-6}
\sum_{V\in \V_1}\sup_{{\bf x}\in V}\text{e}^{{\bf f}_N({\bf x})}<\text{e}^{N(P(T,\mathcal{F},\U)+\frac{\ep}{2})}.
\end{equation}
Observe that  $\cup_{V\in \V_1}V$, which contains $\phi_N(X)$, forms an open subset of $X^N$,   and $\phi_N(X)$  is closed in $X^N$. Consequently, we have
$$s:=d_N((\cup_{V\in \V_1} V)^c, \phi_N(X))=\inf_{{\bf x}\in (\cup_{V\in \V_1} V)^c, {\bf y}\in \phi_N(X)}d_N({\bf x}, {\bf y})>0.$$
Denote $X_{s,N}=\overline{B_{d_N}(\phi_N(X), {s}/{2})}$. It is clear that  $\V_1|_{X_{s,N}}\in \C^o_{X_{s,N}}$. Let $\delta$ be the Lebesgue number of this cover such that $0<\delta<{s}/{2}$.
 By continuity, there exist positive numbers $\delta_0, \delta_1$ with
 $0<\delta_1<\delta_0<\frac{\delta}{N-1}$, such that if $x, y\in X$ with $d(x,y)<\delta_1$, then
 $$\max_{0\le i\le N-1} d(T^i x, T^i y)<\delta_0.$$
Hence,  for any $(N, \delta_1)$-pseudo-orbit $(x_0,x_1,\ldots,x_{N-1})$ in $X^N$,
 we have
 $$ d_N((x_0,x_1,\ldots,x_{N-1}),(x_0,Tx_0,\ldots,T^{N-1}x_{0}))=\max_{0\leq i\leq N-1}d(x_i,T^ix_0)<\delta.$$
This means that  $\V_1$ covers the set of all  $(N, \delta_1)$-pseudo-orbits.

\smallskip
Choose $\delta_1$ smaller if necessary so that
$$
0<\delta_1<\min\left\{\frac14, \frac{\varepsilon}{8(|\U|+\|f_1\|_\infty)}\right\}
$$
where $|\U|$ denotes the cardinality of elements in $\U$,
and
$$-(1-2\delta_1)\log (1-2\delta_1)-2\delta_1\log 2\delta_1<\frac{\ep}{4}.$$
Choose $\delta_2$ such that $0<\delta_2<(\frac{\delta_1}{N})^2$.  Next, for any  given  $n>\frac{N}{\delta_1}$ and  $(n,\delta_2)$-average pseudo-orbit $(x_i)_{i=0}^{n-1}\in X^n_{\delta_2}$,  we will estimate  the number of $(N,\delta_1)$-pseudo-orbits within the sequence $(x_i)_{i=0}^{n-1}$. The method is similar to that used in Claim 3 of \ref{thm:measure-relation:invertible}.
Set
$$A=\{0\leq i\leq n-2:d(Tx_{i},x_{i+1})> \sqrt{\delta_2}\}.$$
By Remark \ref{rem:average-pseudo-orbit} (1), we have  $|A|< (n-1)\sqrt{\delta_2}.$
 Set
 $$
 B=\{0\leq i\leq n-N:(x_i,x_{i+1},\ldots,x_{i+N-1})\text{\ is\ not\ an } (N,\delta_1)\text{-pseudo-orbit}\}.
 $$
 It is easy to see that
$$
B\subset \bigcup_{j\in A}\{j-N+2,\ldots,j-1,j\},
$$
and so $$|B|< (N-1)(n-1)\sqrt{\delta_2}<(n-1)\delta_1.$$
Select $i_1$ as the smallest $i\in[0,n-N]$ such that $i\notin B$. Subsequently,  choose $i_2$  as the smallest $i\in[i_1+N,n-N]$ such that $i\notin B.$  Repeating this procedure inductively, we obtain a finite sequence  $\{i_t\}_{t\in \Lambda}$
such that for the given $(x_i)_{i=0}^{n-1}\in X^n_{\delta_2}$,  there are  mutually disjoint
intervals of length $N$ ,  $\{[i_t,i_{t+N-1}]\}_{t\in\Lambda}$, in $[0,n-1]$ such that their total length $$|\cup_{t\in\Lambda}[i_t,i_{t+N-1}]|>n- [(n-1)\delta_1+N-1 ]>n(1-2\delta_1),$$
and on which,  the corresponding sequences  $\{(x_{i_t},x_{i_{t+1}},\ldots,x_{i_{t+N-1}})\}_{t\in \Lambda}$ are $(N, \delta_1)$-pseudo-orbits.
Furthermore, we observe that for all $(x_i)_{i=0}^{n-1}\in X^n_{\delta_2}$, the possible choices of such  intervals are bounded by $(2n\delta_1+1)C^{[2n\delta_1]}_n$.

Let  $\mathcal{I}=\{[i_t,i_{t+N-1}]\}_{t\in\Lambda}$ be one of such choices.
Define
 $$\V_\mathcal{I}=\cdots\V_1\times \U\times\cdots\times\U\times \V_1\cdots,$$
where the indices corresponding to $\V_1$ are precisely those  intervals  $[i_t,i_{t+N-1}]$,  while the indices for  $\U$  consist solely of  individual elements in the set
$$C=[0,n-1]\setminus \cup_{t\in\Lambda}[i_t,i_{t+N-1}].$$
Since $\V_1\succeq\U^{N},$ it is easy to see that $\V_\mathcal{I}\succeq\U^{n}.$
Let $\V^\prime=\bigcup_{\mathcal{I}}\V_\mathcal{I}$, where the union takes over all the possible choices of $\V_\mathcal{I}$. Then we have $\V^\prime\succeq\U^{n}$. Since $\V_1$ covers the set of all $(N, \delta_1)$-pseudo-orbits, then it follows  that $\V^\prime$ covers $ X^n_{\delta_2}$.
Thus, by definition we have
$$GP_{n,\delta_2} (T,\mathbfcal{F},\U)=\inf\{\sum_{V\in \V}\sup_{{\bf x}\in V}\text{e}^{{\bf f}_n({\bf x})}:\V\in \mathcal{C}_{X^n_{\delta_2}},\V\succeq\U^n\}\leq \sum_{V\in \V^\prime}\sup_{{\bf x}\in V}\text{e}^{{\bf f}_n({\bf x})}.$$
Note that
$$\sum_{V\in \V^\prime}\sup_{{\bf x}\in V}\text{e}^{{\bf f}_n({\bf x})}\leq\sum_{\mathcal{I}}\sum_{V\in \V_\mathcal{I}}\sup_{{\bf x}\in V}\text{e}^{{\bf f}_n({\bf x})}.$$

Fix a choice $I=\{[i_t,i_t+N-1]\}_{t\in\Lambda}$. Then $|\Lambda|\le n/N$ and $|C|<2n\delta_1$.
By block sub-additivity and \eqref{equation-6}, we have
	\begin{align*}
\sum_{V\in \V_\mathcal{I}}\sup_{{\bf x}\in V}\text{e}^{{\bf f}_n({\bf x})} 	& =\sum_{\substack{\{V_t\}_{t\in\Lambda}\subset \V_1,\\ \{U_i\}_{i\in C}\subset\U}}\sup_{{\bf x}\in(\prod_{t\in\Lambda}V_t)\times (\prod_{i\in C}U_i)}\text{e}^{{\bf f}_n({\bf x})}\\
	&\leq\sum_{\substack{\{V_t\}_{t\in\Lambda}\subset \V_1,\\ \{U_i\}_{i\in C}\subset\U}}\sup_{{\bf x}\in(\prod_{t\in\Lambda}V_t)\times (\prod_{i\in C}U_i)}\text{e}^{\sum_{t\in\Lambda}{\bf f}_N((x_{i_t+j})_{j=0}^{N-1})+\sum_{i\in C}{\bf f}_1(x_i)}\\
	&\leq |\U|^{|C|}\cdot\sum_{\{V_t\}_{t\in\Lambda}\subset \V_1}\prod_{t\in\Lambda}\sup_{(x_{i_t+j})_{j=0}^{N-1}\in V_t}\text{e}^{{\bf f}_N((x_{i_t+j})_{j=0}^{N-1})}\cdot \text{e}^{||{\bf f}_1||_\infty\cdot|C|}\\
	&\le |\U|^{2n\delta_1}\cdot\sum_{\{V_t\}_{t\in\Lambda}\subset \V_1}\prod_{t\in\Lambda}\sup_{{\bf x}\in V_t}\text{e}^{{\bf f}_N({\bf x})}\cdot \text{e}^{2n\delta_1||{\bf f}_1||_\infty}\\
	&\overset{(\star)}{=}|\U|^{2n\delta_1}\cdot\prod_{t\in\Lambda}\sum_{V_t\in \V_1}\sup_{{\bf x}\in V_t}\text{e}^{{\bf f}_N({\bf x})}\cdot \text{e}^{2n\delta_1||{\bf f}_1||_\infty}\\
	&\leq|\U|^{2n\delta_1}\cdot (\text{e}^{N(P(T,\mathcal{F},\U)+\frac{\ep}{2})})^{|\Lambda|}\cdot \text{e}^{2n\delta_1||{\bf f}_1||_\infty}\\
	&\leq  |\U|^{2n\delta_1}\cdot \text{e}^{n(P(T,\mathcal{F},\U)+\frac{\ep}{2})}\cdot \text{e}^{2n\delta_1||{\bf f}_1||_\infty}.
	\end{align*}
Here $(\star)$ follows by expanding the product
$$\sum_{i_0,i_1,\dots,i_{\ell-1}=0}^m a_{i_0}a_{i_1}\cdots a_{i_{\ell-1}}=(a_{0}+a_{1}+\cdots+a_{m})^\ell.$$
Therefore,
 \begin{equation*}
	\begin{split}
	GP_{n,\delta_2} (T,\mathbfcal{F},\U)	&\leq\sum_{\mathcal{I}}\sum_{V\in \V_\mathcal{I}}\sup_{{\bf x}\in V}\text{e}^{{\bf f}_n({\bf x})}\\
	&\leq (2n\delta_1+1)C^{[2n\delta_1]}_n\cdot|\U|^{2n\delta_1}\cdot \text{e}^{n(P(T,\mathcal{F},\U)+\frac{\ep}{2})}\cdot \text{e}^{2n\delta_1||{\bf f}_1||_\infty},
	\end{split}
\end{equation*}
It follows from definition that
$$
GP_{\delta_2} (T,\mathbfcal{F},\U)\leq\lim\limits_{n\to\infty}\frac{1}{n}\log C^{[2n\delta_1]}_n+2\delta_1|\U|+P(T,\mathcal{F},\U)+\frac{\ep}{2}+2\delta_1||{\bf f}_1||_\infty.
$$
Note that by Stirling's formula,
$$\lim_{n\rightarrow \infty} \frac{1}{n}\log C_n^{[2n\delta_1]}=-(1-2\delta_1)\log (1-2\delta_1)-2\delta_1\log 2\delta_1.$$
Then,  according to  the choice of $\delta_1$,  we have
$$GP_{\delta_2} (T,\mathbfcal{F},\U)\leq P(T,\mathcal{F},\U)+\ep.$$
Thus,
$$GP (T,\mathbfcal{F},\U)=\inf_{\delta>0}GP_\delta (T,\mathbfcal{F},\U)<P(T,\mathcal{F},\U)+\ep.$$
This completes the proof.
\end{proof}

\section{Proof of Theorem \ref{thm:global-local}}\label{Sect:proof-thm4}

\begin{customthm}{{\bf Theorem \ref{thm:global-local}}}
{\it Let $(X,T)$ be a TDS, and $\mathbfcal{F}=\{{\bf f}_n\}_{n=1}^\infty$ be a block sub-additive potential. Then	
$$GP(T,\mathbfcal{F})=\sup_{\U\in \mathcal{C}^o_X}GP (T,\mathbfcal{F},\U),$$
and for every $\mu\in M(X,T)$,
$$GP_\mu(T,\mathbfcal{F})=\sup_{\U\in \mathcal{C}^o_X}GP_\mu (T,\mathbfcal{F},\U).$$
}
\end{customthm}

\begin{proof}
We only prove that
$$GP(T,\mathbfcal{F})=\sup_{\U\in \mathcal{C}^o_X}GP (T,\mathbfcal{F},\U),$$
 as the proof for the measure-theoretic case is essentially analogous.

\medskip
	LHS $\ge$ RHS: Take $\U\in \mathcal{C}^o_X.$ Let $\ep>0$ be such that the Lebesgue number of $\U$ is larger than $2\ep$.
 For $n\in\N$ and $\delta>0$, choose ${\bf x}_1\in X_\delta^n$  such that
 $$\text{e}^{{\bf f}_n({\bf x}_1)}=\max\limits_{{\bf x}\in X_\delta^n} \text{e}^{{\bf f}_n({\bf x})}.$$
  If $X_\delta^n-B_\ep({\bf x}_1)\neq\emptyset,$ then choose ${\bf x}_2\in X_\delta^n-B_\ep({\bf x}_1)$ such that $$\text{e}^{{\bf f}_n({\bf x}_2)}=\max\limits_{{\bf x}\in X_\delta^n-B_\ep({\bf x}_1)} \text{e}^{{\bf f}_n({\bf x})}.$$
   If $X_\delta^n-(B_\ep({\bf x}_1)\bigcup B_\ep({\bf x}_2))\neq\emptyset,$ then choose ${\bf x}_3\in X_\delta^n-(B_\ep({\bf x}_1)\bigcup B_\ep({\bf x}_2))$ such that
   $$\text{e}^{{\bf f}_n({\bf x}_3)}=\max\limits_{{\bf x}\in X_\delta^n-(B_\ep({\bf x}_1)\bigcup B_\ep({\bf x}_2))} \text{e}^{{\bf f}_n({\bf x})}.$$
 Since $X_\delta^n$ is compact, this procedure terminates after finitely many steps. Consequently, we obtain a subset $\{{\bf x}_1,\ldots,{\bf x}_m\}\subset X_\delta^n$ such that $$\text{e}^{{\bf f}_n({\bf x}_i)}=\max\limits_{{\bf x}\in X_\delta^n-\bigcup_{j=1}^{i-1} B_\ep({\bf x}_j)} \text{e}^{{\bf f}_n({\bf x})},\ i=1,\ldots,m,$$
   with the convention that $\bigcup_{j=1}^{0} B_\ep({\bf x}_j)=\emptyset,$ and $X_\delta^n\subset\bigcup_{j=1}^{m} B_\ep({\bf x}_j).$
Define the sets
 $$V_i=X_\delta^n\cap \Big(B_\ep({\bf x}_i)-\bigcup_{j=1}^{i-1}B_\ep({\bf x}_j)\Big), \ i=1,\ldots,m.$$
Let $\V=\{V_1,\ldots,V_m\}.$
Then $\V\in \mathcal{C}_{X^n_\delta}$ and $\V\succeq\U^n.$
It follows that
$$\sum_{i=1}^m\text{e}^{{\bf f}_n({\bf x}_i)}\ge \sum_{i=1}^m\sup_{{\bf x}\in V_i}\text{e}^{{\bf f}_n({\bf x})}\ge GP_{n,\delta} (T,\mathbfcal{F},\U) .$$
Furthermore, by construction, distinct points of $\{x_1,\ldots,x_m\}$ have $d_n$-distance at least $\varepsilon$. Hence this set is
$(n,\varepsilon/2)$-separated, and
$$
GP_{n,\delta}(T,\mathbfcal{F},\varepsilon/2)
\ge GP_{n,\delta}(T,\mathbfcal{F},\U).
$$
It follows that
$$
GP(T,\mathbfcal{F})\ge \sup_{\U\in \mathcal{C}^o_X}GP (T,\mathbfcal{F},\U).
$$

\medskip
LHS $\le$ RHS: Given $\ep>0$,  select $\U\in \mathcal{C}^o_X$ such that the diameter of $\U$, denoted as  $\diam(\U)=\max_{U\in\U}\diam(U)$, is less than $\ep$.  For $n\in\N$ and $\delta>0$,  let $E$ be an $(n,\ep)$-separated subset of  ${X^n_\delta}$. For any $\V\in \mathcal{C}_{X^n_\delta}$ with $\V\succeq\U^n,$ it is easy to see that each element of $\V$ contains at most one point from $E$.
Consequently,
 $$\sum_{{\bf x}\in E}\text{e}^{{\bf f}_n({\bf x})}\le \sum_{V\in\V}\sup_{{\bf x}\in V}\text{e}^{{\bf f}_n({\bf x})}.$$
Given the arbitrary choices of $E$ and $\V$, we have
$$GP_{n, \delta}(T, \mathbfcal{F}, \ep)\le GP_{n,\delta} (T,\mathbfcal{F},\U).$$
It follows that
$$
GP(T,\mathbfcal{F})\le \sup_{\U\in \mathcal{C}^o_X}GP (T,\mathbfcal{F},\U).
$$
This completes the proof.
\end{proof}

\section*{Acknowledgments}
Research of F.~Cai  was supported by NNSF of China (Grant No.  12301225),  and J.~Li was supported by NNSF of China (Grant No. 12031019).

\end{document}